\documentclass[11pt]{amsart}
\usepackage{amsmath,amssymb,amsthm,mathrsfs,enumerate,bm,xcolor,multirow,pbox}
\usepackage{graphicx,color,framed,tikz}
\allowdisplaybreaks[4]
\numberwithin{equation}{section}
\newcommand{\N}{\mathbb{N}}
\newcommand{\R}{\mathbb{R}}

\newcommand{\E}{\mathbb{E}}
\newcommand{\Prob}{\mathbb{P}}
\newcommand{\G}{\mathbb{G}}

\DeclareMathOperator*{\argmin}{argmin}

\newcommand{\pnorm}[2]{\lVert#1\rVert_{#2}}

\newcommand{\abs}[1]{\lvert#1\rvert}
\newcommand{\bigabs}[1]{\big\lvert#1\big\rvert}
\newcommand{\biggabs}[1]{\bigg\lvert#1\bigg\rvert}

\renewcommand{\epsilon}{\varepsilon}

\renewcommand{\d}[1]{\mathrm{d}#1}

\newcommand{\floor}[1]{\left\lfloor #1 \right\rfloor}

\theoremstyle{definition}\newtheorem{definition}{Definition}
\theoremstyle{remark}

\theoremstyle{remark}\newtheorem{remark}{Remark}
\theoremstyle{definition}\newtheorem{example}{Example}
\theoremstyle{plain}\newtheorem{question}{Question}
\theoremstyle{plain}\newtheorem{theorem}{Theorem}
\theoremstyle{plain}\newtheorem{lemma}{Lemma}
\theoremstyle{plain}\newtheorem{proposition}{Proposition}
\theoremstyle{plain}
\theoremstyle{plain}

\newcounter{tmp}

\begin{document}

\title[Robustness of shape-restricted LSEs]{Robustness of shape-restricted regression estimators: an envelope perspective}
\thanks{Supported in part by NSF Grant DMS-1566514, NI-AID grant R01 AI029168, and by Isaac Newton Institute for Mathematical Sciences, program \emph{Statistical Scalability}, EPSRC Grant Number LNAG/036 RG91310.}

\author[Q. Han]{Qiyang Han}

\address[Q. Han]{
Department of Statistics, Box 354322, University of Washington, Seattle, WA 98195-4322, USA.
}
\email{royhan@uw.edu}

\author[J. A. Wellner]{Jon A. Wellner}

\address[J. A. Wellner]{
Department of Statistics, Box 354322, University of Washington, Seattle, WA 98195-4322, USA.
}
\email{jaw@stat.washington.edu}

\date{\today}

\maketitle

\begin{abstract}
Classical least squares estimators are well-known to be robust with respect to moment assumptions 
concerning the error distribution in a wide variety of finite-dimensional statistical problems;  
generally only a second moment assumption is required for least squares estimators to maintain
the same rate of convergence that they would satisfy if the errors were assumed to be Gaussian.
In this paper, we give a geometric characterization of the robustness of shape-restricted 
least squares estimators (LSEs)  to error distributions with an 
$L_{2,1}$ moment, in terms of the `localized envelopes' of the model.

This envelope perspective gives a systematic approach to proving oracle inequalities 
for the LSEs in shape-restricted regression problems in the random design setting, 
under a minimal $L_{2,1}$ moment assumption on the errors. The canonical isotonic 
and convex regression models, and a more challenging additive regression model 
with shape constraints are studied in detail. Strikingly enough, in the additive model 
both the adaptation and robustness properties of the LSE can be preserved, up to 
error distributions with an $L_{2,1}$ moment, for estimating the shape-constrained 
proxy of the marginal $L_2$ projection of the true regression function. This holds 
essentially regardless of whether or not the additive model structure is correctly specified.

The new envelope perspective goes beyond shape constrained models. 
Indeed, at a general level, the localized envelopes give a sharp characterization 
of the convergence rate of the $L_2$ loss of the LSE between the worst-case 
rate as suggested by the recent work of the authors \cite{han2017sharp}, and the 
best possible parametric rate. 

\end{abstract}


\section{Introduction}

\subsection{Overview}\footnote{See Section \ref{section:notation} for notation.} 
Suppose we observe $(X_1,Y_1),\ldots,(X_n,Y_n)$ from the regression model
\begin{align}\label{regression_model}
Y_i = f_0(X_i)+\xi_i,\quad 1\leq i\leq n.
\end{align}
where the $X_i$'s are independent and identically distributed $\mathcal{X}$-valued 
covariates with law $P$, and the $\xi_i$'s are mean-zero errors independent of $X_i$'s. 
The goal is to recover the true signal $f_0$ based on the observed data $\{(X_i,Y_i)\}_{i=1}^n$.

In the canonical setting where the errors $\xi_i$'s are Gaussian, perhaps the simplest 
estimation procedure for the regression model (\ref{regression_model}) is the  
\emph{least squares estimator} (LSE) $\hat{f}_n$ defined by
\begin{align}\label{lse}
\hat{f}_n \in \argmin_{f \in \mathcal{F}} \sum_{i=1}^n (Y_i-f(X_i))^2,
\end{align}
where $\mathcal{F}$ is a model chosen by the user. The use of the LSE in the Gaussian 
regression model has been theoretically justified in the 1990s and the early 2000s, cf.  \cite{barlett2005local,bartlett2006empirical,birge1993rates,koltchinskii2006local,
koltchinskii2000rademacher,massart2006risk,van1990estimating,van2000empirical,van1996weak}: 

\begingroup
\setcounter{tmp}{\value{theorem}}
\setcounter{theorem}{0} 
\renewcommand\thetheorem{\Alph{theorem}}
\begin{theorem}\label{thm:gaussian_rate}
Suppose that:
\begin{enumerate}
	\item[(E)] the errors $\{\xi_i\}$ are sub-Gaussian (or at least sub-exponential);
	\item[(F)] the model $\mathcal{F}$ satisfies an entropy condition with exponent $ \alpha \in (0,2)$\footnote{$\mathcal{F}$ satisfies an entropy condition with exponent $ \alpha \in (0,2)$ if either (i) $
		\sup_Q \log \mathcal{N}(\epsilon \| F \|_{L_2 (Q)}, \mathcal{F} , L_2 (Q) ) \lesssim \epsilon^{-\alpha}$,
		where the supremum is over all finitely discrete measures $Q$ on $(\mathcal{X},\mathcal{A})$; or (ii) $
		\log \mathcal{N}_{[\, ]} (\epsilon , \mathcal{F}, L_2 (P) ) \lesssim \epsilon^{-\alpha}$.}.
\end{enumerate}
Then
\begin{align}\label{rate_gaussian_error}
\pnorm{\hat{f}_n-f_0}{L_2(P)}=\mathcal{O}_{\mathbf{P}}\big(n^{-\frac{1}{2+\alpha}}\big).
\end{align}
\end{theorem}
\endgroup
\setcounter{theorem}{\thetmp}

Furthermore, the rate (\ref{rate_gaussian_error}) is unimprovable under the entropy 
conditions (F) in a minimax sense, see e.g. \cite{yang1999information}. 

Although the condition (F) is widely accepted in the literature as a complexity 
measurement of the model $\mathcal{F}$, it is far from clear if the light-tailed 
condition on the errors (E) is necessary for the theory. Recently, we showed \cite{han2017sharp} 
that the condition (E) is actually more than a mere technicality: 
\begingroup
\setcounter{tmp}{\value{theorem}}
\setcounter{theorem}{1} 
\renewcommand\thetheorem{\Alph{theorem}}
\begin{theorem}\label{thm:general_rate}
	Suppose that condition (E) in Theorem \ref{thm:gaussian_rate} is replaced by
	\begin{enumerate}
		\item[(E')] the errors $\{\xi_i\}$ have a finite $L_{p,1}$ moment ($p\geq 1$)
	\end{enumerate}
	and (F) holds. Then
	\begin{align}\label{rate_general_error}
	\pnorm{\hat{f}_n-f_0}{L_2(P)}=\mathcal{O}_{\mathbf{P}}\big(n^{-\frac{1}{2+\alpha}}\vee n^{-\frac{1}{2}+\frac{1}{2p}}\big).
	\end{align}
\end{theorem}
\endgroup
\setcounter{theorem}{\thetmp}
We also showed \cite{han2017sharp} that the rate (\ref{rate_general_error}) cannot be improved under (F) alone. 
Comparing with (\ref{rate_gaussian_error}), the rate in (\ref{rate_general_error}) clearly indicates that if the 
model $\mathcal{F}$ only satisfies (F), the best possible moment condition on the errors to guarantee the 
same rate of convergence of the LSE as in the case of Gaussian errors is $p\geq 1+2/\alpha$.

The starting point for this paper originates from a remarkable result due to Cun-Hui Zhang \cite{zhang2002risk} 
in the context of isotonic regression. Zhang \cite{zhang2002risk} showed that the $L_2$ loss of the isotonic 
LSE achieves the usual worst-case (minimax)
 $\mathcal{O}_{\mathbf{P}}(n^{-1/3})$ rate, and the adaptive rate 
$\mathcal{O}_{\mathbf{P}}(\sqrt{\log n/n})$ if the true signal is, say, $f_0$ equals a constant, 
under only a second moment 
assumption on the errors.  

We view the first of these two properties established by Zhang 
as a ``robustness property'' of the LSE with respect to the distribution of the errors $\{ \xi_i \}$. 
We formalize this with the following definition:

\begin{definition} 
\label{defn:Robustness}
We will say that the estimator sequence $\{ \hat{f}_n \}$ is \emph{$L_2$-robust with respect to the errors 
$\{\xi_i$\} in the model ${\mathcal F}$} (or just $L_2$-robust), if $\hat{f}_n$ converges to $f_0$ in $L_2(P)$ at the same rate 
for zero mean $0$ errors with $ \| \xi_i \|_2 < \infty$ as for errors $\{ \xi_i \}$ that are Gaussian or sub-Gaussian.
Similarly, if the same rate holds for zero mean errors with $\| \xi_i \|_{2,1} < \infty$, we say that 
$\{ \hat{f}_n \}$ is \emph{$L_{2,1}$-robust with respect to the errors  $\{\xi_i$\} in the model ${\mathcal F}$.}
\end{definition}

Similarly, we view the second of the two properties established by Zhang as an ``adaptation property'' of the 
LSE with respect to the model ${\mathcal F}$:

\begin{definition}
\label{defn:Adaptation} 
We will say that the estimator sequence $\{ \hat{f}_n\}$ is  \emph{adaptive to a subset ${\mathcal G}_m$
of the model ${\mathcal F}$} if it achieves a nearly (up to factors of $\log n$) parametric rate of convergence at
all points $f \in {\mathcal G}_m$.
\end{definition}

For the shape-constrained models we consider here the subsets ${\mathcal G}_m $ of ${\mathcal F}$ are natural subclasses 
of extreme points of the class ${\mathcal F}$:  in the isotonic model ${\mathcal F}$ the collections ${\mathcal G}_m$
consisting of $m$ constant non-decreasing pieces, and in the convex regression model ${\mathcal G}_m$ 
can be taken to be the piecewise linear (convex) functions with at most $m$ linear pieces.

Zhang's work \cite{zhang2002risk} has generated intensive research interest in further 
understanding the adaptation properties of the isotonic and other shape-restricted LSEs in 
recent years, cf. 
\cite{bellec2018sharp,chatterjee2015risk,chatterjee2015adaptive,guntuboyina2013global,guntuboyina2017nonparametric}. 
These papers share a common theme: the shape-restricted LSEs 
are adaptive to certain subsets $\{\mathcal{G}_m\}$ of the model ${\mathcal F}$ 
under a (sub-)gaussian assumption on the distribution of the errors in the regression model. 

Despite substantial progress in the adaptation properties of various shape-restricted LSEs, 
there remains little progress in further understanding their $L_2$-robustness properties beyond the 
isotonic model studied by Zhang \cite{zhang2002risk}. Indeed, the challenges involved 
here were noted in Guntuboyina and Sen \cite{guntuboyina2017nonparametric} (page 30) as follows:
``......\emph{However the existing proof techniques for these risk bounds strongly rely on the 
assumption of sub-Gaussianity. It will be very interesting to prove risk bounds in these problems 
without Gaussianity. We believe that new techniques will need to be developed for this}''. 
One of the goals of this paper is to provide new approaches and insights concerning  the $L_2$
(or $L_{2,1}$)-robustness of various shape-restricted LSEs.

Initially we had hoped to study this problem by appealing to the general 
Theorem \ref{thm:general_rate}. However, the theory in Theorem \ref{thm:general_rate} requires at least a 
third moment (note that here $\alpha=1$ for the isotonic model). This implies that the isotonic shape constraint 
must contain more information than that provided by the entropic structure alone, 
so that Theorem \ref{thm:general_rate} fails to 
fully capture the $L_2$-robustness of the isotonic LSE. 

One particular useful feature of the isotonic model is an explicit min-max formula for 
the isotonic LSE in terms of partial sum processes; see e.g. \cite{robertson1988order}. 
Zhang's techniques \cite{zhang2002risk} make full use of the min-max representation, 
and are therefore substantially of an analytic flavor. Similar techniques have also been 
used in \cite{chatterjee2015risk,gao2017minimax}, but have apparently not yet successful
 in dealing with any other shape constrained models. The rigidity in this analytic approach 
 naturally motivates the search for other `softer' properties of the isotonic shape constrained 
 model that explain the robustness of the LSE. These considerations lead to the following question.

\begin{question}\label{question:geometric_feature_fcn_class}
What geometric aspects of the isotonic shape constrained model give rise to the 
$L_2$(or $L_{2,1}$)-robustness property of the LSE?
\end{question}

To put this question into a more general setting, note that Theorem \ref{thm:general_rate} 
implies that the LSE can converge as slowly as $\mathcal{O}_{\mathbf{P}}(n^{-1/4})$ 
for certain hard models when the errors only have a second moment, 
while in the aforementioned isotonic regression case, it is possible that the LSE converges 
at a nearly parametric rate $\mathcal{O}_{\mathbf{P}}(\sqrt{\log n/n})$ 
for certain special isotonic functions. Therefore it seems more promising to search for a characterization of the convergence 
rate of the $L_2$ loss of the LSE in terms of some geometric feature of the model 
$\mathcal{F}$, when the errors have only an $L_2$(or $L_{2,1}$) moment.

The first main contribution of this paper is to shed light on Question \ref{question:geometric_feature_fcn_class} 
from an `envelope' perspective at this general level. Roughly speaking, the size of the `localized envelopes' 
of the model $\mathcal{F}$ determines the convergence rate of the $L_2$ loss of the LSE when the 
errors only have an $L_{2,1}$ moment. More specifically, let $F_0(\delta)$ be the envelope for 
$\mathcal{F}_0(\delta)\equiv \{f \in \mathcal{F}_0: Pf^2\leq \delta^2\}$ where $\mathcal{F}_0\equiv \mathcal{F}-f_0$. 
We show that (cf. Theorem \ref{thm:envelope_rate_upper_bound}), under a certain 
uniform entropy condition on the function class, if for some $0\leq \gamma\leq 1$, the localized envelopes have the growth rate
\begin{align}\label{cond:size_envelope_intro}
\pnorm{F_0(\delta)}{L_2(P)} \sim \delta^\gamma:
\end{align}
then the convergence rate of the LSE in the $L_2$ loss is no worse than 
\begin{align}\label{rate:envelope_size}
\mathcal{O}_{\mathbf{P}}\big(n^{-\frac{1}{2(2-\gamma)}}\big).
\end{align}
Furthermore, the rate (\ref{rate:envelope_size}) cannot be improved under the 
condition (\ref{cond:size_envelope_intro}), cf. Theorem \ref{thm:envelope_rate_lower_bound}. 
It is easily seen  from (\ref{rate:envelope_size}) that, as the size of the 
localized envelopes increases, the rate of the $L_2$ loss of the LSE deteriorates 
from the parametric rate $\mathcal{O}_{\mathbf{P}}(n^{-1/2})$ to the worst-case rate 
$\mathcal{O}_{\mathbf{P}}(n^{-1/4})$ as suggested by Theorem \ref{thm:general_rate}. 
For isotonic regression, we will see that the localized envelopes of the model are 
small in the sense that $\gamma \approx 1$ (up to logarithmic factors) when $f_0=0$, 
and hence the LSE converges at a nearly parametric rate under an $L_{2,1}$ moment 
assumption on the errors. For the hard models identified in \cite{han2017sharp} 
(cf. Example \ref{ex:hard_regression} below), the localized envelopes are big in the 
sense that $\gamma=0$ so the LSE can only converge at the worst-case rate.

Addressing Question \ref{question:geometric_feature_fcn_class} from a geometric 
point of view is not only of interest in its own right, but also serves as an important step in 
better understanding the robustness properties of other shape constrained models. 
This is the context of the second main contribution of this paper: we aim at improving our understanding of the $L_{2,1}$-robustness property of shape restricted LSEs, 
by providing a systematic approach to proving oracle inequalities in the random design 
regression setting for these LSEs under an $L_{2,1}$ moment condition on the errors. This goal is achieved by exploiting the idea of small envelopes from the solution to 
Question \ref{question:geometric_feature_fcn_class}. The formulation of the oracle inequality 
follows its fixed-design counterparts that highlight the automatic rate-adaptive behavior 
of the LSE, cf. \cite{bellec2018sharp,chatterjee2015risk}. More specifically, we first 
prove the following oracle inequality that holds for the canonical isotonic and convex 
LSEs in the simple regression models (cf. Theorem \ref{thm:isotonic_reg}): 
Suppose that $\pnorm{f_0}{\infty}<\infty$ and the errors $\{\xi_i\}$ are i.i.d. mean-zero 
with $\pnorm{\xi_1}{2,1}<\infty$. Then for any $\delta \in (0,1)$, there exists some 
constant $c>0$ such that with probability $1-\delta$,
\begin{align}\label{oracle_ineq_intro}
\pnorm{\hat{f}_n-f_0^\ast}{L_2(P)}^2 \leq c \inf_{m \in \N}  
\left(\inf_{f_m \in \mathcal{G}_m}\pnorm{f_m-f_0^\ast}{L_2(P)}^2+\frac{m}{n}\cdot \log^{2} n\right),
\end{align}
where $f_0^\ast$ is the $L_2(P)$-projection of $f_0$ onto the space of square 
ntegrable monotonic non-decreasing (resp. convex) functions, and $\mathcal{G}_m$ 
is the class of piecewise constant non-decreasing (resp. linear convex) functions on 
$[0,1]$ with at most $m$ pieces in the isotonic (resp. convex) model. 
The oracle inequality (\ref{oracle_ineq_intro}) is further verified for the shape-restricted 
LSEs in the additive model (cf. Theorem \ref{thm:additive_isotonic_reg}), where now 
$f_0$ is the marginal $L_2$ projection of the true regression function.
One striking message of the oracle inequality for the shape-restricted LSEs in the 
additive model is the following: both the adaptation and $L_{2,1}$-robustness properties 
of the LSE can be preserved, up to error distributions with an $L_{2,1}$ moment, for 
estimating the shape-constrained proxy of the marginal $L_2$ projection of the true 
regression function, \emph{essentially regardless of whether or not the additive structure is correctly specified}.

The proofs in this paper rely heavily on the new empirical process tools and proof 
techniques developed in \cite{han2017sharp}. Although we will list relevant results, 
readers are referred to \cite{han2017sharp} for more discussion of the new tools. 
Along the way we also resolve the stochastic boundedness issue of convexity 
shape-restricted LSEs at the boundary, which may be of independent interest 
(this problem is in fact an open problem in the field, cf. \cite{guntuboyina2017nonparametric}).

\subsection{Notation}\label{section:notation}
For a real-valued random variable $\xi$ and $1\leq p<\infty$, let $\pnorm{\xi}{p} := \big(\E\abs{\xi}^p\big)^{1/p} $ 
denote the ordinary $p$-norm. The $L_{p,1}$ norm for a random variable $\xi$ is defined by 
\begin{align*}
\pnorm{\xi}{p,1}:=\int_0^\infty {\Prob(\abs{\xi}>t)}^{1/p}\ \d{t}.
\end{align*}
It is well known that $L_{p+\epsilon}\subset L_{p,1}\subset L_{p}$ holds for any underlying 
probability measure, and hence a finite $L_{p,1}$ condition requires slightly more than a $p$-th moment, 
but no more than any $p+\epsilon$ moment, see Chapter 10 of \cite{ledoux2013probability}. 
In this paper, we will primarily be concerned with the case $p=2$.

For a real-valued measurable function $f$ defined on $(\mathcal{X},\mathcal{A},P)$, $\pnorm{f}{L_p(P)}\equiv \big(P\abs{f}^p)^{1/p}$ denotes the usual $L_p$-norm under $P$, and $\pnorm{f}{\infty}\equiv\pnorm{f}{L_\infty}\equiv  \sup_{x \in \mathcal{X}} \abs{f(x)}$. $f$ is said to be $P$-centered if $Pf=0$. $L_p(g,B)$ denotes the $L_p(P)$-ball centered at $g$ with radius $B$. For simplicity we write $L_p(B)\equiv L_p(0,B)$. 

Let $(\mathcal{F},\pnorm{\cdot}{})$ be a subset of the normed space of real functions $f:\mathcal{X}\to \R$. Let $\mathcal{N}(\epsilon,\mathcal{F},\pnorm{\cdot}{})$ be the $\epsilon$-covering number, and let  $\mathcal{N}_{[\,]}(\epsilon,\mathcal{F},\pnorm{\cdot}{})$ be the $\epsilon$-bracketing number; see page 83 of \cite{van1996weak} for more details. To avoid unnecessary measurability digressions, we assume that $\mathcal{F}$ is countable throughout the article. As usual, for any $\phi:\mathcal{F}\to \R$, we write $\pnorm{\phi(f)}{\mathcal{F}}$ for $ \sup_{f \in \mathcal{F}} \abs{\phi(f)}$. 

Throughout the article $\epsilon_1,\ldots,\epsilon_n$ will be i.i.d. Rademacher random variables independent of all other random variables. $C_{x}$ will denote a generic constant that depends only on $x$, whose numeric value may change from line to line unless otherwise specified. $a\lesssim_{x} b$ and $a\gtrsim_x b$ mean $a\leq C_x b$ and $a\geq C_x b$ respectively, and $a\asymp_x b$ means $a\lesssim_{x} b$ and $a\gtrsim_x b$ [$a\lesssim b$ means $a\leq Cb$ for some absolute constant $C$]. For two real numbers $a,b$, $a\vee b\equiv \max\{a,b\}$ and $a\wedge b\equiv\min\{a,b\}$. We slightly abuse notation by defining $\log(x)\equiv \log(x\vee e)$.

\subsection{Organization}
Section \ref{section:theory} is devoted to a treatment of the relationship between the size of the localized envelopes and the convergence rate of the $L_2$ loss of the least squares estimator. Section \ref{section:shape_restricted} is devoted to applications to shape-restricted regression problems. Proofs are deferred to Sections \ref{section:proof_main_results} and \ref{section:proof_remaining}.

\section{Convergence rate of the LSE: the envelope characterization}\label{section:theory}

\subsection{Upper and lower bounds}

Our first main result is the following.

\begin{theorem}\label{thm:envelope_rate_upper_bound}
Suppose that $\xi_1,\ldots,\xi_n$ are i.i.d. mean-zero errors independent of i.i.d. covariates 
$X_1,\ldots,X_n$ with law $P$ such that $\pnorm{\xi_1}{2,1}<\infty$. Further suppose that 
$\mathcal{F}_0\equiv \mathcal{F}-f_0$ is a VC-subgraph class, and the envelopes 
$F_0(\delta)$ of $\mathcal{F}_0(\delta)\equiv \{f \in \mathcal{F}_0: Pf^2\leq \delta^2 \}$ satisfy the growth condition
\begin{align}\label{cond:size_envelope}
\pnorm{F_0(\delta)}{L_2(P)} \leq c \cdot \delta^\gamma,\quad \textrm{for all } \delta>0
\end{align}
for some constants $0\leq \gamma \leq 1$ and $c>0$. If $\pnorm{\hat{f}_n-f_0}{\infty}=\mathcal{O}_{\mathbf{P}}(1)$, then
\begin{align*}
\pnorm{\hat{f}_n-f_0}{L_2(P)}= \mathcal{O}_{\mathbf{P}}\big( n^{-\frac{1}{2(2-\gamma)}}\big).
\end{align*}
\end{theorem}

\begin{remark}\label{rmk:upper_bound_lse}
Some technical remarks are in order.
\begin{enumerate}
\item If instead of $\pnorm{\hat{f}_n-f_0}{\infty}=\mathcal{O}_{\mathbf{P}}(1)$
it is assumed that $\mathcal{F}_0 \subset L_\infty(1)$, then the conclusion of 
Theorem \ref{thm:envelope_rate_upper_bound} can be strengthened to an expectation: 
$\E\pnorm{\hat{f}_n-f_0}{L_2(P)}= \mathcal{O}\big( n^{-\frac{1}{2(2-\gamma)}}\big)$.
\item Condition (\ref{cond:size_envelope}) on the size of the localized envelopes can 
be modified to incorporate logarithmic factors. In particular, if
\begin{align*}
\pnorm{F_0(\delta)}{L_2(P)} \leq c \cdot \delta^\gamma \log^\tau (1/\delta),
\end{align*}
then we may slightly modify the proof of Theorem \ref{thm:envelope_rate_upper_bound} 
to see that the convergence rate of the $L_2$ loss of the LSE is given by 
\begin{align*}
\mathcal{O}_{\mathbf{P}}\big(n^{-\frac{1}{2(2-\gamma)}} \log^{\frac{\tau}{2-\gamma}} n\big).
\end{align*}
\item We assume that the errors are identically distributed for simplicity: the case of mean-zero, 
independent but not necessarily identically distributed errors follows from a minor modification of the proof.
\end{enumerate}
\end{remark}

\begin{remark}\label{rmk:uniform_VC}
Theorem \ref{thm:envelope_rate_upper_bound} is actually proved for $\mathcal{F}_0$ 
under a more general \emph{uniform VC-type} condition: $\mathcal{F}_0$ is said to be of 
uniform VC-type if there exists some $\alpha \in [0,2)$ and $\beta \in [0,\infty)$
\footnote{We can also allow $\alpha=2, \beta<-2$ but we are not aware of any such examples.} 
such that for any probability measure $Q$, and any $\epsilon \in (0,1), \delta>0$,
\begin{align}\label{cond:uniform_entropy}
\log \mathcal{N}\left(\epsilon\pnorm{F_0(\delta)}{L_2(Q)},\mathcal{F}_0(\delta),L_2(Q)\right)\lesssim \epsilon^{-\alpha}\log^\beta(1/\epsilon).
\end{align}
The most significant examples for uniform VC-type classes are the VC-subgraph 
classes ($\alpha=0,\beta=1$).  Other important examples include the VC-major classes, 
which satisfy (\ref{cond:uniform_entropy}) up to a logarithmic factor (cf. Lemma \ref{lem:subset_VC}). 
As we will see in Section \ref{section:shape_restricted}, the canonical examples of VC-major 
classes that satisfy (\ref{cond:uniform_entropy}) considered in this paper are the classes of 
bounded monotonic non-decreasing and convex functions on $[0,1]$.
\end{remark}

\begin{remark}
From a purely probabilistic point of view, the condition (\ref{cond:size_envelope}) is related to 
Alexander's capacity function \cite{alexander1985rates,alexander1987central,alexander1987rates} 
defined for VC class of sets that gives relatively sharp asymptotic local moduli of weighted 
empirical processes indexed by such classes. Results in a similar vein can be found in 
\cite{gine2006concentration} who generalized this notion to bounded VC-subgraph function classes.
\end{remark}

So far we have derived an upper bound for the convergence rate of the $L_2$ loss of the LSE 
under the condition (\ref{cond:size_envelope}). It is natural to wonder if such an upper bound is sharp 
in an appropriate sense.

\begin{theorem}\label{thm:envelope_rate_lower_bound}
Let $P$ be the uniform distribution on $[0,1]$. For any $\gamma \in (0,1]$, there exists some 
uniformly bounded VC-subgraph class  $\tilde{\mathcal{F}}$ on $[0,1]$ and some $f_0 \in \tilde{\mathcal{F}}$ 
such that $\tilde{\mathcal{F}}_0\equiv \tilde{\mathcal{F}}-f_0$ satisfies (\ref{cond:size_envelope}), 
and the following property holds: for each $\epsilon \in (0,1/2)$, there exist some constants 
$ c_{\epsilon,\gamma}>0$, $\mathfrak{p}>0$ and some law for $\xi_1$ with $\pnorm{\xi_1}{2(1-\epsilon)}<\infty$ 
such that, for $n$ large enough depending on $\epsilon,\gamma$, there exists a LSE $f_n^\ast$ whose $L_2$ loss satisfies
\begin{align*}
 \pnorm{f^\ast_n-f_0}{L_2(P)}\geq c_{\epsilon,\gamma}\cdot n^{-\frac{1}{2(2-\gamma)}-c_{\gamma}'\epsilon}
\end{align*}
with probability at least $\mathfrak{p}>0$. The constant $c_{\gamma}'$ can be taken to be $2/\gamma$.
\end{theorem}

Theorem \ref{thm:envelope_rate_lower_bound} shows that our upper bound Theorem 
\ref{thm:envelope_rate_upper_bound} cannot be improved substantially under (\ref{cond:size_envelope}): 
the size of the localized envelopes drives the convergence rate of the $L_2$ loss 
of the LSE over VC-subgraph models (or more generally, models of uniform VC-type) 
in the heavy-tailed regression setting where the errors only admit (roughly) 
a second moment. Since the median regression estimator over VC-subgraph 
models achieves a nearly parametric rate $\mathcal{O}_{\mathbf{P}}(\sqrt{\log n/n})$ 
at least when the errors are symmetric and admit smooth densities; cf. Section 3.4.4 of \cite{van1996weak}, 
Theorem \ref{thm:envelope_rate_lower_bound} illustrates a genuine deficiency of the LSE in 
VC-subgraph models when the envelopes of the model are not small.
We remark that the case $\gamma=0$ is excluded mainly for simplicity of presentation; 
similar conclusions hold under a slightly weaker formulation, cf. Theorem 5 of \cite{han2017sharp}.

The proofs of Theorems \ref{thm:envelope_rate_upper_bound} and 
\ref{thm:envelope_rate_lower_bound} are based on recent developments 
on the \emph{equivalence} between the convergence rate of the $L_2$ 
loss of the LSE and the size of the multiplier empirical process, cf. 
\cite{chatterjee2014new,han2017sharp,van2015concentration}. 
For the upper bound, our proofs rely heavily on a new multiplier inequality 
developed in \cite{han2017sharp}. The lower bound, on the other hand, is 
based on an explicit construction of $\tilde{\mathcal{F}}$ that witnesses the 
desired rate within uniformly bounded VC-subgraph classes satisfying (\ref{cond:size_envelope}).

\subsection{Examples}\label{section:examples_revisit}
In this section, we use Theorem \ref{thm:envelope_rate_upper_bound} to examine the 
convergence rate of the $L_2$ loss of the LSE in several important examples. 

\begin{example}[Linear model]\label{ex:linear_regression}
	Let $\mathcal{F}\equiv \{f_\beta(x)\equiv \beta^\top x: \beta \in \R^d\}$ and let $P$ 
	be the uniform distribution on $[0,1]^d$ . This is the simplest linear regression model. 
	A second moment assumption on the errors $\xi_i$'s yields a closed-form LSE with a 
	parametric convergence rate: 
	$\pnorm{\hat{f}_n-f_0}{L_2(P)}\asymp \pnorm{\hat{\beta}_n-\beta_0}{2}=\mathcal{O}_{\mathbf{P}}(n^{-1/2})$. 
	This rate is obviously much faster than the worst-case rate $\mathcal{O}_{\mathbf{P}}(n^{-1/4})$ as suggested by 
	Theorem \ref{thm:general_rate}. Thus the LSE sequence $\{\hat{f}_n \}$ is $L_2$-robust for the model 
	${\mathcal F}$ by a direct argument while our Theorem 1 very nearly recovers this:  it shows that 
	$\{\hat{f}_n \}$ is $L_{2,1}$-robust for the model  ${\mathcal F}$.
	
	For simplicity of discussion, we assume $d=1$ in the sequel. We may also restrict the model to be 
	$\{f_\beta:\beta \in [-1,1]\}$; this is viable since the LSE \emph{localizes} in the sense that 
	$\pnorm{\hat{f}_n}{\infty}=\abs{\hat{\beta}_n}=\mathcal{O}_{\mathbf{P}}(1)$. Moreover, it is clear 
	that the model is a VC-subgraph class. For any $\delta>0$, $\pnorm{f_\beta}{L_2(P)}\leq \delta$ 
	implies that $\abs{\beta}\leq \sqrt{3} \delta$, and thus
	\begin{align*}
	F(\delta)(x)=\sup_{\beta \in [-\sqrt{3} \delta,\sqrt{3} \delta]} \abs{\beta x}= \sqrt{3} \delta \abs{x},
	\end{align*}
	which in turn yields $\pnorm{F(\delta)}{L_2(P)}=\delta$.  
	Hence Theorem \ref{thm:envelope_rate_upper_bound} applies with $\gamma=1$ to recover the 
	usual parametric rate $\mathcal{O}_{\mathbf{P}}(n^{-1/2})$ for the $L_2$ loss of the LSE. 
	
	Our approach here should be compared with the common practice of using local entropy to recovery the exact parametric rate for parametric models---but the latter does not extend directly to the heavy-tailed regression setting, cf. pages 152-153 of \cite{van2000empirical}.
\end{example}

\begin{example}[Isotonic model]\label{ex:isotonic_regression}
	Let $\mathcal{F}$ be the class of monotonic non-decreasing functions on $[0,1]$ 
	and let $P$ be the uniform distribution on $[0,1]$. It is shown in a related fixed design 
	setting (cf. \cite{chatterjee2015risk,gao2017minimax,zhang2002risk}) that a second 
	moment condition on the errors $\xi_i$ is sufficient for the isotonic LSE to achieve the 
	nearly parametric adaptive rate $\mathcal{O}_{\mathbf{P}}(\sqrt{\log n/n})$ in the 
	discrete $\ell_2$ loss, when the true signal is $f_0=0$. This naturally suggests a 
	similar rate for the $L_2$ loss of the isotonic LSE in the random design setting. 
	Apparently, this (suggested) nearly parametric rate is far from the worst-case rate 
	$\mathcal{O}_{\mathbf{P}}(n^{-1/4})$.
	
	In this model, since the univariate isotonic LSE localizes in $L_\infty$ norm 
	(cf. Lemma \ref{lem:uniform_bound_iso}), we may assume without loss of generality 
	that $\mathcal{F}\equiv \{f: \textrm{non-decreasing}, \pnorm{f}{\infty}\leq 1\}$. 
	The entropy condition (\ref{cond:uniform_entropy}) can be verified using the VC-major 
	property of $\mathcal{F}$ up to a logarithmic factor (cf. Lemma \ref{lem:subset_VC}). 
	On the other hand, for any $\delta>0$, by monotonicity and the $L_2$ constraint, we can take
	\begin{align*}
	F(\delta)(x) \equiv \delta \cdot \big(x^{-1/2}\vee (1-x)^{-1/2}\big) \wedge 1.
	\end{align*}
	Evaluating the integral we see that $\pnorm{F(\delta)}{L_2(P)}\lesssim \delta \sqrt{\log (1/\delta)}$. 
	Then an application of Theorem \ref{thm:envelope_rate_upper_bound} along with 
	Remarks \ref{rmk:upper_bound_lse} (2) and \ref{rmk:uniform_VC}, we see that the $L_2$ 
	loss of the LSE $\hat{f}_n$ converges at a parametric rate up to logarithmic factors when the truth $f_0$ is a constant function and the errors are $L_{2,1}$.
	The observation concerning the role of
	the localized envelopes in the isotonic model here is the starting point for a systematic development of oracle inequalities for shape-restricted LSEs in Section \ref{section:shape_restricted}.
\end{example}

\begin{example}[Single change-point model]
Let $\mathcal{F}\equiv \{\bm{1}_{[a,1]}: a \in [0,1]\}$ be the model containing signals on $[0,1]$ 
with a single change point. Let $P$ be the uniform distribution on $[0,1]$. 

This model is contained in the isotonic model---from here we already know by Example 
\ref{ex:isotonic_regression} that the localized envelopes of $\mathcal{F}$ are small, and 
hence the LSE converges at a rate no worse than a nearly parametric rate under an $L_{2,1}$ moment assumption on the errors. We can do better: since the localized envelopes are 
exactly given by $F(\delta)=\bm{1}_{[1-\delta^2,1]}$, it follows that $\pnorm{F(\delta)}{L_2(P)}=\delta$, 
and hence by Theorem \ref{thm:envelope_rate_upper_bound} with $\gamma=1$ we see that the 
LSE converges exactly at the parametric rate $\mathcal{O}_{\mathbf{P}}(n^{-1/2})$ even if the 
errors only admit an $L_{2,1}$ moment. This is in stark contrast with the \emph{multiple change-points model} detailed below.
\end{example}

\begin{example}[Multiple change-points model]\label{ex:hard_regression}
	Consider the following multiple change-points model:
	\begin{align*}
	\mathcal{F}_k\equiv \bigg\{&\sum_{i=1}^k c_i \bm{1}_{[x_{i-1},x_{i}]}: \abs{c_i}\leq 1,\\
	& \quad 0\leq x_0<x_1<\ldots<x_{k-1}<x_k\leq 1\bigg\}, k\geq 1.
	\end{align*}
	It is shown in \cite{han2017sharp} that the $L_2$ loss of the LSE over (a subset of) 
	$\mathcal{F}_k$ cannot converge at a rate faster than $\mathcal{O}_{\mathbf{P}}(n^{-1/4})$ 
	for some errors $\xi_i$ with only (roughly) a second moment. The LSE fails to be rate-optimal 
	in this model: if the errors are Gaussian (or even bounded), the convergence rate of the $L_2$ 
	loss of the LSE (over VC-subgraph classes) is no worse than $\mathcal{O}_{\mathbf{P}}(\sqrt{\log n/n})$. 
	
	Note that in this model, the localized envelopes are given by $F(\delta)\equiv 1$ for any 
	$\delta>0$ and hence $\pnorm{F(\delta)}{L_2(P)}=1$. Applying Theorem \ref{thm:envelope_rate_upper_bound} 
	with $\gamma=0$ recovers the correct  rate $\mathcal{O}_{\mathbf{P}}(n^{-1/4})$ for the $L_2$ loss of the 
	LSE in this model.	
\end{example}

\begin{example}[Unimodal model]
	Let $\mathcal{F}$ contain all (bounded) unimodal functions on $[0,1]$, i.e. all $f :[0,1]\to \R$ 
	such that there exists some $x^\ast \in [0,1]$ with $f|_{[0,x^\ast]}$ non-decreasing and 
	$f|_{[x^\ast,1]}$ non-increasing. \cite{chatterjee2015adaptive} and \cite{bellec2018sharp}
	considered the performance 
	of the LSE in a fixed-design unimodal Gaussian regression setting, where similar adaptive 
	behavior as in the isotonic case (cf. \cite{zhang2002risk}) is derived. Since the class of (bounded) 
	unimodal functions on $[0,1]$ contains the class of multiple change-points model $\mathcal{F}_1$ 
	as studied in Example \ref{ex:hard_regression}, our results here 
	imply that \emph{the 
	unimodal shape constraint does not inherit} the $L_2$ (or $L_{2,1}$)-robustness property as in the isotonic shape constraint in 
	Example \ref{ex:isotonic_regression}: the worst-case $\mathcal{O}_{\mathbf{P}}(n^{-1/4})$ is 
	attained by the LSE in the unimodal regression model for some errors $\xi_i$'s with (roughly) a second moment.
\end{example}

\section{Shape-restricted regression problems}\label{section:shape_restricted}

As briefly mentioned in the Introduction, it is well-known that in the fixed design 
regression setting, the isotonic least squares estimator (LSE) only requires a second 
moment condition on the errors to enjoy an oracle inequality, cf. \cite{chatterjee2015risk,gao2017minimax,zhang2002risk}. 
The proof techniques used therein rely crucially on (i) some form of representation of the isotonic 
LSE in terms of partial sum processes, and (ii) martingale inequalities. 
Unfortunately, such an explicit representation does not exist beyond the isotonic LSE, 
and hence these techniques do not readily extend to other problems. 

Our goal here is to give a systematic treatment of the robustness properties of shape-restricted 
LSEs in a random design setting, up to error distributions with an $L_{2,1}$ moment. 
The examples we examine are (i) the canonical isotonic and convex regression models, 
and (ii) additive regression models with monotonicity and convexity shape constraints. 
As we will see, the `smallness' of the localized envelopes, along with their special 
geometric properties, play a central role in our approach.

Henceforth, the isotonic (resp. convex) model refers to the regression model based 
on the class of monotonic non-decreasing (resp. convex) functions on $[0,1]$.

\subsection{Prologue: the canonical problems}\label{section:isotonic_reg}

We start by considering the `canonical' problems in the area of shape restricted regression: 
the isotonic and convex regression problems. Note that a generic LSE $\hat{f}_n$ in (\ref{lse}) 
is only well-defined on the design points $X_1,\ldots,X_n$. Our results below hold for the 
\emph{canonical LSEs}: for the isotonic (respectively convex) model, $\hat{f}_n$ is defined to be 
the unique left-continuous piecewise constant (resp. linear) function on $[0,1]$ with jumps (respectively kinks) 
at (potentially a subset of)  $\{\hat{f}_n(X_i)\}_{i=1}^n$.

Some further notation: let $\mathcal{M}_m\equiv \mathcal{M}_m([0,1])$ (respectively 
$\mathcal{C}_m\equiv \mathcal{C}_m([0,1])$) be the class of all non-decreasing piecewise 
constant functions (respectively convex piecewise linear functions) on $[0,1]$ with at most $m$ pieces. 
Let $P$ denote the uniform distribution on $[0,1]$ for simplicity of exposition.

\begin{theorem}\label{thm:isotonic_reg}
Consider the regression model (\ref{regression_model}). 
Let $\mathcal{F}$ be 
either the isotonic or convex model. Suppose that $\pnorm{f_0}{\infty}<\infty$, 
and the errors are i.i.d. mean-zero with $\pnorm{\xi_1}{2,1}<\infty$. 
Then for any $\delta \in (0,1)$, there exists $c\equiv c(\delta,\pnorm{\xi}{2,1},\pnorm{f_0}{\infty},\mathcal{F})>0$ 
such that with probability $1-\delta$, the canonical LSE $\hat{f}_n$ defined above satisfies
\begin{align*}
\pnorm{\hat{f}_n-f_0^\ast}{L_2(P)}^2 \leq c \inf_{m \in \N}  
\left(\inf_{f_m \in \mathcal{G}_m}\pnorm{f_m-f_0^\ast}{L_2(P)}^2+\frac{m}{n}\cdot \log^{2} n \right),
\end{align*}
where $f_0^\ast = \argmin_{g \in \mathcal{F}\cap L_2(P)} \pnorm{f_0-g}{L_2(P)}$, 
and $\mathcal{G}_m=\mathcal{M}_m$ for the isotonic model and $\mathcal{G}_m=\mathcal{C}_m$ 
for the convex model. 
\end{theorem}

The isotonic regression problem, included here mainly for sake of later development 
in the additive model, is a benchmark example in the family of shape-restricted regression 
problems. Even in this simplest case, the above oracle inequality in $L_2(P)$ loss seems 
new\footnote{An oracle inequality in $L_2(\Prob_n)$ loss follows immediately from  
\cite{chatterjee2015risk} (with a second moment assumption on the errors) since  the 
monotone cone does \emph{not} change with the design points. See \cite{han2017isotonic} 
for different techniques in the multivariate isotonic regression problem when the errors are Gaussian. }.

For the more interesting convex regression problem, our oracle inequality here confirms for the 
first time both the adaptation and robustness properties of the convex LSE up to error distributions 
with an $L_{2,1}$ moment. Previous oracle inequalities for the convex LSE exclusively focused 
on the fixed-design setting under a (sub-)Gaussian assumption on the errors \cite{bellec2018sharp,chatterjee2015risk}; 
see also Section 3 of \cite{guntuboyina2017nonparametric} for a review. 
	
\begin{remark}
Two technical comments on  the formulation of the oracle inequality in Theorem \ref{thm:isotonic_reg}:
\begin{enumerate}
\item The oracle inequality holds for the projection $f_0^\ast$ of $f_0$ to $\mathcal{F}\cap L_2(P)$ 
and hence allows for model mis-specification: the only assumption on $f_0$ is boundedness: 
$\pnorm{f_0}{\infty}<\infty$. The same comment also applies to the oracle inequality in the additive model below.
\item The oracle inequality cannot be strengthened to an expectation, in view of a counterexample 
discovered in \cite{balazs2015near} in the convex model: the convex LSE $\hat{f}_n$ has infinite 
$L_2$ risk in estimating $f_0=0$ even if the errors are bounded: $\E \pnorm{\hat{f}_n-0}{L_2(P)}=\infty$.
\end{enumerate}
\end{remark}

\definecolor{color1}{RGB}{102,0,51}
\definecolor{color2}{RGB}{153,0,76}
\definecolor{color3}{RGB}{204,0,102}
\definecolor{color4}{RGB}{255,0,127}
\definecolor{color5}{RGB}{255,51,153}

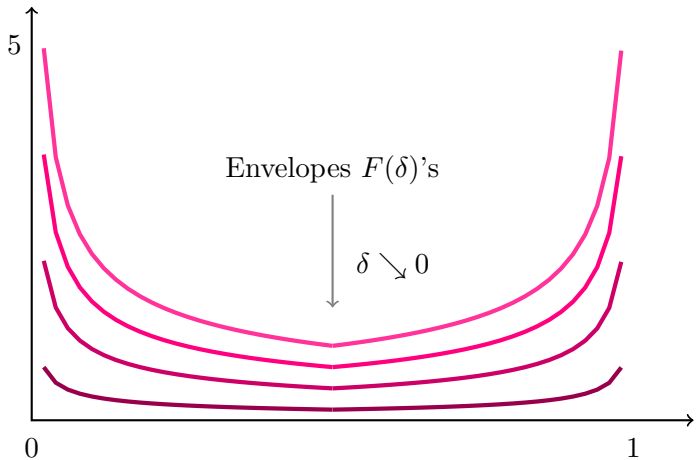
\begin{figure}
	\begin{tikzpicture}[xscale=8,yscale=1]
	\draw [thick, <->] (0,5.5) -- (0,0) -- (1.1,0);
	\draw[color5, ultra thick, domain=0.02:0.5] plot (\x, {0.7/sqrt(\x)} );
	\draw[color5, ultra thick, domain=0.5:0.98] plot (\x, {0.7/sqrt(1-\x)} );
	\draw[color4, ultra thick, domain=0.02:0.5] plot (\x, {0.5/sqrt(\x)} );
	\draw[color4, ultra thick, domain=0.5:0.98] plot (\x, {0.5/sqrt(1-\x)} );
	\draw[color3, ultra thick, domain=0.02:0.5] plot (\x, {0.3/sqrt(\x)} );
	\draw[color3, ultra thick, domain=0.5:0.98] plot (\x, {0.3/sqrt(1-\x)} );
	\draw[color2, ultra thick, domain=0.02:0.5] plot (\x, {0.1/sqrt(\x)} );
	\draw[color2, ultra thick, domain=0.5:0.98] plot (\x, {0.1/sqrt(1-\x)} );
	\node [above] at (0.6,1.75) {$\delta \searrow 0$};
	\node [above] at (0.5,3) {Envelopes $F(\delta)$'s};
	\node [below] at (1,-0.1) {1};
	\node [below] at (0,-0.1) {0};
	\node [left] at (0,5) {5};
	\draw [gray, thick, ->] (0.5,3) -- (0.5,1.5);
	\end{tikzpicture}
	\caption{Envelopes for isotonic model with $c=1$ in (\ref{eqn:envelope}). From top to bottom: $\delta=0.7, 0.5, 0.3, 0.1$. }
	\label{fig: envelope_fcn}
\end{figure}

\subsubsection{Proof strategy of Theorem \ref{thm:isotonic_reg}}
The proof of Theorem \ref{thm:isotonic_reg} contains two major steps.
\begin{enumerate}
\item[\textbf{(Step 1)}] We first localize the shape-restricted LSEs in $L_\infty$ norm. 
This step requires some understanding of the boundary behavior of the shape-restricted 
LSEs under a second moment assumption on the errors. The case for isotonic regression 
is relatively straightforward, while the case for convex regression is much more difficult. 
Here we resolve this issue in Lemma \ref{lem:uniform_bound_iso}. 
\item[\textbf{(Step 2)}] After the localization in Step 1, the problem essentially reduces to 
controlling a multiplier empirical process of the form
\begin{align}\label{multiplier_ep_lse}
\E \sup_{f \in \mathcal{F}: f-f_0^\ast \in L_2(\delta_n)\cap L_\infty(B) }\biggabs{\frac{1}{\sqrt{n}}\sum_{i=1}^{n} \xi_i (f-f_0^\ast)(X_i)}.
\end{align}
A sharp bound for (\ref{multiplier_ep_lse}) is inspired by the observation in 
Example \ref{ex:isotonic_regression}, where the (untruncated) localized envelopes of the isotonic model take the form
\begin{align}\label{eqn:envelope}
F(\delta)(x)\equiv c\delta \cdot \big(x^{-1/2}\vee (1-x)^{-1/2}\big)
\end{align}
for some absolute constant $c>0$. The envelopes for the convex model also take the same form 
(\ref{eqn:envelope}), cf. Lemma \ref{lem:convex_1dim_envelope}. On the other hand, the 
localized envelopes (\ref{eqn:envelope}) are centered at $0$, while the multiplier empirical 
process (\ref{multiplier_ep_lse}) in question is centered at $f_0^\ast$. By exploiting the exact 
form of (\ref{eqn:envelope}), we perform a `change-of-center argument' on (\ref{multiplier_ep_lse}) 
by shifting $f_0^\ast$ to an arbitrary piecewise simple signal 
$f_m \in \mathcal{G}_m \cap L_\infty(\pnorm{f_0^\ast}{\infty})$, cf. Lemma \ref{lem:l_2_shape_constraint_mep_est}, 
thereby reducing the control of (\ref{multiplier_ep_lse}) to control of several multiplier 
empirical processes centered at $0$. The effect of the heavy-tailed $\xi_i$'s is then accounted for, 
via the multiplier inequality developed in \cite{han2017sharp}, by a uniform estimate for the 
corresponding empirical processes in terms of the $L_2$ size of the localized envelopes (\ref{eqn:envelope}). 
\end{enumerate}

\begin{remark}
Currently our oracle inequality comes with a $\log^{2} n$ term. It is known in (i) the fixed 
design isotonic model with a second moment assumption on the errors, and (ii) the fixed 
design convex model with 
sub-Gaussian errors, that the power of the logarithmic factor 
can be reduced to $1$. The additional logarithmic factor in Theorem \ref{thm:isotonic_reg} 
occurs due to the use of VC-major property for the isotonic and convex models in the random 
design setting: the entropy estimate of bounded VC-major classes comes with logarithmic 
factors that involve the $L_2$ size of the envelopes (cf. Lemma \ref{lem:subset_VC}). 
\end{remark}

\subsection{Additive regression model with shape constraints}\label{section:additive_isotonic_reg}

Consider fitting $(x,z)\mapsto \phi_0(x,z)$, the conditional mean of the regression model
\begin{align}\label{additive_model}
Y_i = \phi_0(X_i,Z_i)+\xi_i,\quad 1\leq i\leq n,
\end{align}
by additive models of the form $\{(x,z)\mapsto f(x)+h(z)\}_{f \in \mathcal{F}, h\in \mathcal{H}}$, 
where $\mathcal{F},\mathcal{H}$ are two function classes on $[0,1]$. To capture the mathematical 
essence of the problem, we assume that the covariates $\{(X_i,Z_i)\}_{i=1}^n$ are i.i.d. from 
the uniform law $P$ on $[0,1]^2$ and are independent of the errors $\{\xi_i\}$. 
We use $P_X,P_Z$ to denote the marginal distributions of $P$. For identifiability 
we assume that $\mathcal{H}$ is centered.

Additive models of the type have a long history; see e.g.  \cite{hastie1990generalized,stone1985additive}. 
When the additive model is well specified (i.e. $\phi_0(x,z)=f_0(x)+h_0(z)$ with $f_0 \in \mathcal{F}, h_0\in \mathcal{H}$), 
and the nonparametric components enjoy smoothness assumptions, standard methods such as iterative backfitting, 
e.g. \cite{mammen1999existence} and penalized LSE (smooth spline), e.g. \cite{wahba1990spline}, 
can be used to estimate $f_0$ and $h_0$. 

Instead of computational issues, we will be interested here in certain structural aspects of the additive LSE $\hat{f}_n$ defined via:
\begin{align}\label{lse_additive_model}
(\hat{f}_n,\hat{h}_n) \in \argmin_{(f,h)\in \mathcal{F}\times \mathcal{H}} \sum_{i=1}^n\big(Y_i-f(X_i)-h(Z_i)\big)^2.
\end{align}
Since the true regression function $\phi_0$ need not have an additive structure, 
one may naturally expect that $\hat{f}_n$ and $\hat{h}_n$ estimate the marginal 
$L_2$ projections $x\mapsto f_0(x)\equiv P_Z \phi_0(x,Z)$ and 
$z \mapsto h_0(z)\equiv P_X \phi_0(X,z) - P \phi_0$ (cf. Appendix 4, page 439 of \cite{bickel1998efficient}). 
Our primary \emph{structural question} on the behavior of the additive LSE $\hat{f}_n$ 
concerns the situation in which the model $\mathcal{F}$ involves shape constraints:
\begin{question}\label{question:additive_model_structure}
Does the additive LSE $\hat{f}_n$ over the shape constrained model $\mathcal{F}$ 
enjoy similar robustness and adaptation properties as in the univariate case 
(treated in Theorem \ref{thm:isotonic_reg})?
\end{question}

The next theorem gives an affirmative answer to Question \ref{question:additive_model_structure}.

\begin{theorem}\label{thm:additive_isotonic_reg}
Suppose that $(X_i, Z_i, Y_i )$, $i=1,\ldots , n$, are i.i.d. with values 
in $[0,1]\times [0,1] \times \R$ and satisfy 
(\ref{additive_model}) where $\pnorm{\phi_0}{\infty} < \infty$, and the errors $\{ \xi_i \}$ are i.i.d. 
mean zero with $\pnorm{\xi_1}{2,1}<\infty$. Let $\mathcal{F}$ be either the isotonic or convex model. 
Further suppose that $\mathcal{H}\subset L_\infty(2\pnorm{\phi_0}{\infty})$ 
satisfies the following $L_\infty$ covering bound: for some $\gamma \in (0,2)$
\begin{align}\label{ineq:cond_H2_1}
\log \mathcal{N}(\epsilon, \mathcal{H},L_\infty)\lesssim \epsilon^{-\gamma},\textrm{ for all }\epsilon \in (0,1).
\end{align}
Then for any $\delta \in (0,1)$, there exists 
$c\equiv c(\delta,\pnorm{\xi}{2,1},\pnorm{\phi_0}{\infty}, \mathcal{F}, \mathcal{H})>0$ such that with probability 
$1-\delta$, the canonical LSE $\hat{f}_n$ in (\ref{lse_additive_model}) satisfies
\begin{align*}
\pnorm{\hat{f}_n-f_0^\ast}{L_2(P)}^2 \leq c \inf_{m \in \N}  \left(\inf_{f_m \in \mathcal{G}_m}
\pnorm{f_m-f_0^\ast}{L_2(P)}^2+\frac{m}{n}\cdot \log^{2} n \right),
\end{align*}
where $f_0^\ast = \argmin_{g \in \mathcal{F}\cap L_2(P)} \pnorm{f_0-g}{L_2(P)}$ with 
$f_0=P_Z\phi_0(\cdot,Z)$, and $\mathcal{G}_m=\mathcal{M}_m$ for the isotonic model and 
$\mathcal{G}_m=\mathcal{C}_m$ for the convex model.
\end{theorem}

There is very limited theoretical understanding of the properties of shape-restricted estimators when additive models are used. \cite{meyer2013semi} investigated identifiability 
issue for the additive LSE in the fixed design setting. \cite{mammen2007additive} considered 
\emph{pointwise} performance of the LSE where both $\mathcal{F}$ and $\mathcal{H}$ are 
monotonic with errors admitting exponential moments. \cite{cheng2009semiparametric} 
gives an extension to a semiparametric setting assuming the same moment condition on 
the errors, still considering pointwise performance of the LSEs for the isotonic components. 
\cite{chen2016generalized} proved consistency of the MLEs for a generalized class of additive 
and index models with shape constraints, without rate considerations. A common feature of all 
these works is that the model is required to be well-specified.

To the best knowledge of the authors, Theorem \ref{thm:additive_isotonic_reg} is the first oracle 
inequality for shape-restricted LSEs in regression using an additive model, and moreover, 
allowing for model mis-specification: not only the regression function class $\mathcal{F}$ can 
be mis-specified, but the additive model itself may also be mis-specified.  Our result here therefore 
gives a strong positive answer to Question \ref{question:additive_model_structure}: both the adaptation 
and robustness properties of additive shape-restricted LSEs can be preserved in estimating the shape 
constrained proxy of the marginal $L_2$ projection of the true regression function,  up to error 
distributions with an $L_{2,1}$ moment, \emph{essentially regardless of whether or not the 
additive structure is correctly specified}.  

\subsubsection{Examples under correct specification of the additive structure}
Now we consider the important situation when $\phi_0$ has an additive structure:
\begin{align*}
\phi_0(x,z)\equiv f_0(x)+ h_0(z).
\end{align*}
In such a scenario, our result here is related to the recent work \cite{van2015penalized}, 
who asserted that the rate optimality nature of the (penalized) LSE over $\mathcal{F}$ 
in the Gaussian regression setting can be preserved regardless of the smoothness 
level of $\mathcal{H}$. Our Theorem \ref{thm:additive_isotonic_reg} reveals a further 
structural property of the LSEs: the robustness and adaptation merits due to shape 
constraints can also be preserved, regardless of the choice of $\mathcal{H}$ under the entropy condition (\ref{ineq:cond_H2_1}).

To further illustrate this point, we consider some examples. 
\begin{itemize}
\item (\emph{Parametric model}) $\mathcal{H}\equiv \{f_\beta(z)\equiv \beta (z-1/2): \beta \in [-1,1]\}$. 
In this case (\ref{additive_model}) becomes the semiparametric partially linear model. 
\item (\emph{Smooth model}) $\mathcal{H}$ is the class of centered uniformly bounded 
$\alpha$-H\"older ($\alpha>1/2$) continuous functions on $[0,1]$ with uniformly bounded 
derivatives (cf. Theorem 2.7.1 of \cite{van1996weak}). 
\item (\emph{Shape constrained model}) $\mathcal{H}$ is the class of centered uniformly  
Lipschitz convex functions on $[0,1]$ (cf. Corollary 2.7.10 of \cite{van1996weak}). 
\end{itemize}

\subsubsection{Proof strategy of Theorem \ref{thm:additive_isotonic_reg}}

The basic strategy in our proof of Theorem \ref{thm:additive_isotonic_reg} is similar to that of 
Theorem \ref{thm:isotonic_reg}. First, we need to localize the LSEs in $L_\infty$ norm 
under a second moment assumption on the errors and $P_Z H^2<\infty$, cf. Lemma 
\ref{lem:uniform_bound_additive_iso}. Next, in addition to the multiplier empirical process 
(\ref{multiplier_ep_lse}), the major additional empirical process we need to control is
\begin{align}\label{multiplier_ep_lse_additive}
&\E \sup_{ \substack{f \in \mathcal{F}: f-f_0^\ast \in L_2(\delta_n)\cap L_\infty(B)\\ h \in \mathcal{H}} }
    \biggabs{\frac{1}{\sqrt{n}}\sum_{i=1}^{n} \epsilon_i (f-f_0^\ast)(X_i)(h-(\phi_0-f_0))(X_i,Z_i)}.
\end{align}
where the $\epsilon_i$'s are i.i.d. Rademacher random variables. 
One notable feature in (\ref{multiplier_ep_lse_additive}) is that the supremum over $\mathcal{H}$ 
need \emph{not} be localized when the interest is in the behavior of $\hat{f}_n$, cf. Proposition 
\ref{prop:additive_lse}. In other words, no apriori information on the behavior of $\hat{h}_n$ 
(other than the assumption (\ref{ineq:cond_H2_1}))  is needed in order to understand the behavior of $\hat{f}_n$. 

The entropy condition (\ref{ineq:cond_H2_1}) serves as a sufficient condition for a sharp estimate for 
(\ref{multiplier_ep_lse_additive}) (and thereby for the oracle inequality in Theorem \ref{thm:additive_isotonic_reg}), 
but is apparently not necessary; we make such a choice here to cover the above common examples. 
A case-by-case study is possible as long as (\ref{multiplier_ep_lse_additive}) can be well-controlled. 
For instance, it is not hard to verify a similar bound for (\ref{multiplier_ep_lse_additive}) as in 
Lemma \ref{lem:additive_mep_est} (and hence the oracle inequality for shape-restricted LSEs 
$\hat{f}_n$) when the additive structure is correctly specified, and $\mathcal{H}$ is the class of 
centered indicator functions over closed intervals on $[0,1]$ and $h_0=0$ (note that this class 
fails to satisfy (\ref{ineq:cond_H2_1}) since $\mathcal{H}$ is not totally bounded in $L_\infty$). 
This is a difficult case: although the $L_2$ loss of the LSE $\hat{h}_n$ is known to converge 
at a worst-case rate $\mathcal{O}_{\mathbf{P}}(n^{-1/4})$ (cf. Example \ref{ex:hard_regression}), 
Theorem \ref{thm:additive_isotonic_reg} tells us that the bad behavior of $\hat{h}_n$ has no 
effect on the good (robust and adaptive) performance of $\hat{f}_n$, at least under reasonable 
assumption on the distribution of the covariates $(X,Z)$.

\section{Proofs of the main results}\label{section:proof_main_results}

In this section we outline the main steps in proving the main results of the paper, namely:

\begin{enumerate}
	\item Theorems \ref{thm:envelope_rate_upper_bound} and \ref{thm:envelope_rate_lower_bound} 
	characterizing the geometric feature of the model that determines the actual 
	convergence rate of the $L_2$ loss of the least squares estimator, and
	\item Theorems \ref{thm:isotonic_reg} and \ref{thm:additive_isotonic_reg} 
	highlighting oracle inequalities in shape restricted regression models with a $L_{2,1}$ moment assumption on the errors.
\end{enumerate}

Proofs of many technical intermediate results will be deferred to Section \ref{section:proof_remaining}.

\subsection{Preliminaries}

In this subsection we collect the empirical process tools that will be needed in the proofs to follow. 
Our first ingredient is a sharp multiplier inequality proved in \cite{han2017sharp}. 

\begin{lemma}[Theorem 1 in \cite{han2017sharp}]\label{lem:multiplier_ineq}
Suppose that $\xi_1,\ldots,\xi_n$ are i.i.d. mean-zero random variables independent of i.i.d. 
$X_1,\ldots,X_n$. Let $\mathcal{F}_1\supset \cdots \supset \mathcal{F}_n$ be a 
non-increasing sequence of function classes.  Assume further that there exist 
non-decreasing concave functions $\{\psi_n\}:\R_{\geq 0}\to \R_{\geq 0}$ with 
$\psi_n(0)=0$ such that
\begin{align}\label{cond:local_maximal_mulep}
\E \bigg\lVert \sum_{i=1}^k\epsilon_i f(X_i)\bigg\lVert_{\mathcal{F}_k}\leq \psi_n(k)
\end{align}
holds for all $1\leq k\leq n$. Then
\begin{align*}
\E \bigg\lVert \sum_{i=1}^n \xi_i f(X_i)\bigg\lVert_{\mathcal{F}_n}\leq 4\int_0^\infty \psi_n\big(n\cdot \Prob(\abs{\xi_1}>t)\big)\ \d{t}.
\end{align*}
\end{lemma}

Lemma \ref{lem:multiplier_ineq} controls the first moment of the multiplier empirical process. 
For higher moments, the following moment inequality is useful.

\begin{lemma}[Proposition 3.1 of \cite{gine2000exponential}]\label{lem:p_moment_estimate}
Suppose $X_1,\ldots,X_n$ are i.i.d. with law $P$ and $\xi_1,\ldots,\xi_n$ are i.i.d. mean-zero 
random variables with $\pnorm{\xi_1}{2}<\infty$. Let $\mathcal{F}$ be a class of measurable 
functions such that $\sup_{f \in \mathcal{F}} Pf^2\leq \sigma^2$. Then for any $q\geq 1$,
\begin{align*}
\E \sup_{f \in \mathcal{F}}\bigg\lvert\sum_{i=1}^n \xi_i f(X_i)\bigg\rvert^q
&\leq K^q\bigg[\bigg(\E\sup_{f \in \mathcal{F}} \bigg\lvert\sum_{i=1}^n \xi_i f(X_i)\bigg\rvert \bigg)^q\\
&\qquad+q^{q/2}(\sqrt{n}\pnorm{\xi_1}{2}\sigma)^q +q^q\E\max_{1\leq i\leq n} \abs{\xi_i}^q\sup_{f \in \mathcal{F}}\abs{f(X_i)}^q\bigg].
\end{align*}
Here $K>0$ is a universal constant.
\end{lemma}

To use Lemma \ref{lem:multiplier_ineq}, we need to control the size of the empirical process. Let
\begin{align}\label{def:uniform_entropy}
J(\delta,\mathcal{F},L_2) \equiv   \int_0^\delta  \sup_Q\sqrt{1+\log \mathcal{N}\big(\epsilon\pnorm{F}{L_2(Q)},\mathcal{F},L_2(Q)\big)}\ \d{\epsilon}
\end{align}
denote the \emph{uniform entropy integral}, where the supremum is taken over all 
discrete probability measures.

We will frequently use the following Koltchinskii-Pollard maximal inequality.

\begin{lemma}[Theorem 2.14.1 of \cite{van1996weak}]\label{lem:KP_maximal_ineq}
Let $\mathcal{F}$ be a class of measurable functions with measurable envelope $F$, 
and $X_1,\ldots,X_n$ are i.i.d. random variables with law $P$. Then
\begin{align*}
\E \bigg\lVert \sum_{i=1}^n \epsilon_i f(X_i)\bigg\lVert_{\mathcal{F}} \lesssim \sqrt{n}J(1,\mathcal{F},L_2) \pnorm{F}{L_2(P)}.
\end{align*}
\end{lemma}

Our last technical ingredient is Talagrand's concentration inequality \cite{talagrand1996new} for the empirical process in the form given by \cite{massart2000constants}:

\begin{lemma}\label{lem:talagrand_conc_ineq}
Let $\mathcal{F}$ be a class of measurable functions such that $\sup_{f \in \mathcal{F}} \pnorm{f}{\infty}\leq b$. Then
\begin{align*}
\Prob\bigg(\sup_{f \in \mathcal{F}}\abs{\G_n f} \geq 2\E\sup_{f \in \mathcal{F}}\abs{\G_n f} +\sqrt{8\sigma^2 x}+34.5 b \frac{x}{\sqrt{n}} \bigg)\leq e^{-x},
\end{align*}
where $\sigma^2\equiv \sup_{f \in \mathcal{F}} \mathrm{Var}_P f$, and $\mathbb{\G}_n\equiv \sqrt{n}(\Prob_n-P)$.
\end{lemma}

\subsection{Proof of Theorem \ref{thm:envelope_rate_upper_bound}}

\begin{proof}[Proof of Theorem \ref{thm:envelope_rate_upper_bound}]
	We only prove the case $\mathcal{F}_0\subset L_\infty(1)$ as in Remark \ref{rmk:upper_bound_lse} (1). 
	The proof for the case
	$\pnorm{\hat{f}_n-f_0}{\infty}=\mathcal{O}_{\mathbf{P}}(1)$ follows with only 
	minor 
	modifications. We also work with the more general uniform VC-type condition as in Remark \ref{rmk:uniform_VC}. 
	Let $\delta_n\equiv n^{-\frac{1}{2(2-\gamma)}}$. By the proof of Proposition 2 of \cite{han2017sharp}, 
	we only need to estimate for each $t\geq 1$, with $\mathcal{F}_0(r)=\{f \in \mathcal{F}-f_0: \pnorm{f}{L_2(P)}\leq r\}$,
	\begin{align*}
	\E \bigg(\sup_{f \in \mathcal{F}_0(2^j t \delta_n)}\bigg\lvert\frac{1}{\sqrt{n}}
	   \sum_{i=1}^n \xi_i f(X_i)\bigg\rvert\bigg)^2,\quad  \E \bigg(\sup_{f \in \mathcal{F}_0(2^j t \delta_n)} 
	   \bigg\lvert\frac{1}{\sqrt{n}}\sum_{i=1}^n \epsilon_i f^2(X_i)\bigg\rvert\bigg)^2.
	\end{align*}
	By the contraction principle for Rademacher processes and the moment inequality 
	Lemma \ref{lem:p_moment_estimate}, we only need to estimate the sum of 
	\begin{align}\label{ineq:envelope_rate_upper_bound_1}
	(I)\equiv \bigg(\E\sup_{f \in \mathcal{F}_0(2^j t \delta_n)} \bigg\lvert\frac{1}{\sqrt{n}} 
	    \sum_{i=1}^n \xi_i f(X_i)\bigg\rvert\bigg)^2 + \bigg(\E\sup_{f \in \mathcal{F}_0(2^j t \delta_n)}  
	    \bigg\lvert \frac{1}{\sqrt{n}}\sum_{i=1}^n \epsilon_i f(X_i)\bigg\lvert \bigg)^2
	\end{align}
	and 
	\begin{align}\label{ineq:envelope_rate_upper_bound_2}
	(II)\equiv \big(2^j t \delta_n(\pnorm{\xi_1}{2}\vee 1)\big)^2 +  n^{-1}\cdot 
	\E \max_{1\leq i\leq n} (\abs{\xi_i}\vee 1)^2 \cdot \pnorm{F_0(2^j t \delta_n)}{L_2(P)}^2.
	\end{align}
	For the first summand (\ref{ineq:envelope_rate_upper_bound_1}), by the 
	Koltchinskii-Pollard maximal inequality for empirical processes (cf. Lemma \ref{lem:KP_maximal_ineq}), since $\mathcal{F}$ is of uniform VC-type, it follows that
	\begin{align*}
	\max_{1\leq k\leq n} \E\sup_{f \in \mathcal{F}_0(2^j t \delta_n)} \bigg\lvert\frac{1}{\sqrt{k}}\sum_{i=1}^k \epsilon_i f(X_i)\bigg\lvert \leq C_{\mathcal{F}} \pnorm{F_0(2^j t \delta_n)}{L_2(P)}\leq C_{\mathcal{F}}' (2^j t)^{\gamma} \delta_n^\gamma.
	\end{align*}
	We may apply the multiplier inequality Lemma \ref{lem:multiplier_ineq} with $
	\psi_n(k)\equiv \sqrt{k} C_{\mathcal{F}}' (2^j t)^{\gamma} \delta_n^\gamma$ 
	to see that
	\begin{align*}
	\E\sup_{f \in \mathcal{F}_0(2^j t \delta_n)} \bigg\lvert\frac{1}{\sqrt{n}}\sum_{i=1}^n \xi_i f(X_i)\bigg\lvert &\leq 4C_{\mathcal{F}}' (2^j t)^\gamma \pnorm{\xi_1}{2,1} \delta_n^\gamma.
	\end{align*}
	Hence,
	\begin{align}\label{ineq:envelope_rate_upper_bound_3}
	(\ref{ineq:envelope_rate_upper_bound_1})\leq C_{\mathcal{F},\xi} (2^j t\delta_n)^{2\gamma}.
	\end{align}
	(\ref{ineq:envelope_rate_upper_bound_2}) is easy to handle by noting that $\E \max_{1\leq i\leq n} (\abs{\xi_i}\vee 1)^2\lesssim n$ under the assumption that $\pnorm{\xi_1}{2}<\infty$, which entails that
	\begin{align}\label{ineq:envelope_rate_upper_bound_4}
	(\ref{ineq:envelope_rate_upper_bound_2})\leq C_{\xi} \left( (2^j t \delta_n)^2+(2^j t\delta_n)^{2\gamma} \right).
	\end{align}
	Combining (\ref{ineq:envelope_rate_upper_bound_3}) and (\ref{ineq:envelope_rate_upper_bound_4}) and the arguments in the proof of Proposition 2 of \cite{han2017sharp}, we have
	\begin{align*}
	\Prob\big(\pnorm{\hat{f}_n-f_0}{L_2(P)}\geq t \delta_n\big)&\leq C_{\mathcal{F},\xi} \sum_{j\geq 0: 2^jt \delta_n\leq 2} \frac{  (2^j t \delta_n)^2+ (2^j t\delta_n)^{2\gamma}. }{   \left(2^{2j} t^2 \sqrt{n}\delta_n^2\right)^2 }\\
	&\leq C_{\mathcal{F},\xi}' \big( n\delta_n^{2(2-\gamma)}\big)^{-1} \sum_{j\geq 0}  \frac{1}{(2^jt)^{4-2\gamma} }\leq C_{\mathcal{F},\xi}'' t^{-2},
	\end{align*}
	where the last inequality follows from the choice of $\delta_n$. Now the claim of the theorem (in the form of Remark \ref{rmk:upper_bound_lse} (1) and under the more general condition as in Remark \ref{rmk:uniform_VC}) follows by integrating the above tail estimate.
\end{proof}

\subsection{Proof of Theorem \ref{thm:envelope_rate_lower_bound}}

The basic device we will use to derive a lower bound for the risk of the least squares estimator is the following.

\begin{proposition}[Proposition 6 of \cite{han2017sharp}]\label{prop:proof_route_lower_bound_lse}
	Let
	\begin{align*}
	F_n(\delta)\equiv \sup_{f \in \mathcal{F}-f_0: Pf^2\leq \delta^2} (\Prob_n-P)(2\xi f-f^2)-\delta^2\equiv E_n(\delta)-\delta^2.
	\end{align*}
	Suppose that $0<\delta_1<\delta_2$ are such that $E_n(\delta_1)<F_n(\delta_2)$. Then there exists a LSE $f^\ast_n$ such that $\pnorm{f^\ast_n-f_0}{L_2(P)}\geq \delta_1$.
\end{proposition}

The key ingredient in applying the above device is the following.

\begin{proposition}\label{prop:lower_bound_envelope}
	For any $\gamma \in (0,1]$, there exists some VC-subgraph class  $\tilde{\mathcal{F}}$ satisfying (\ref{cond:size_envelope}) with the following property: for each $\epsilon \in (0,1/2)$, there exists some law for $\xi_1$ with $\pnorm{\xi_1}{2(1-\epsilon)}<\infty$ such that 
	\begin{enumerate}
		\item for any $\vartheta \geq 4$, there exists some $\mathfrak{p}>0$, with $\delta_2 \equiv \vartheta n^{-\frac{1}{2(2-\gamma)}}$,
		\begin{align*}
		\Prob\bigg(F_n(\delta_2)\geq \frac{1}{2}c_1\vartheta^{\gamma} n^{-\frac{1 }{2-\gamma}}\tau_n(\epsilon,\gamma) \bigg)\geq 2\mathfrak{p},
		\end{align*}
		holds for $n$ large enough depending on $\epsilon, \vartheta,\gamma$.  
		Here $c_1$ depends on $\epsilon,\gamma$, and $\tau_n(\epsilon,\gamma)\equiv n^{\frac{1-\gamma}{2-\gamma}\cdot \frac{\epsilon}{2-\epsilon}}$.
		\item for any $\rho>0$, with $\delta_1 \equiv \rho n^{-\frac{1}{2(2-\gamma)}-\beta_\epsilon}$,
		\begin{align*}
		\Prob\left(E_n(\delta_1)\leq \mathfrak{p}^{-1} C_{\epsilon,\xi} \rho^{\gamma} n^{-\frac{1 }{2-\gamma}}\omega_n(\epsilon,\gamma)\right)\geq 1-\mathfrak{p}.
		\end{align*}
		 Here  $\omega_n(\epsilon,\gamma)= n^{ -\gamma \beta_\epsilon +\frac{\epsilon}{2(1-\epsilon)}}$.
	\end{enumerate}
	In (1)-(2) above, $F_n(\delta)\equiv \sup_{f \in \tilde{\mathcal{F}}: Pf^2\leq \delta^2} (\Prob_n -P)(2\xi f-f^2)-\delta^2\equiv E_n(\delta)-\delta^2$.
\end{proposition}

The proof of Proposition \ref{prop:lower_bound_envelope} relies on a delicate construction of a tree-structured $\tilde{\mathcal{F}}$, and a sequence of technical arguments including concentration of empirical processes, the Paley-Zygmund moment argument, and an exact characterization of the size of the maxima of summations. To ease reading, a formal proof of Proposition \ref{prop:lower_bound_envelope} will be given in Section \ref{section:proof_remaining}. 

\begin{proof}[Proof of Theorem \ref{thm:envelope_rate_lower_bound}]
	Let $f_0=0$. In order to apply Proposition \ref{prop:proof_route_lower_bound_lse}, we only need to require an order in the exponent of $\tau_n(\cdot,\cdot)$ and $\omega_n(\cdot,\cdot)$ in Proposition \ref{prop:lower_bound_envelope}, by making a good choice of $\beta_\epsilon$. To this end, it suffices to require
	\begin{align*}
	-\gamma \beta_\epsilon +\frac{\epsilon}{2(1-\epsilon)}< \frac{1-\gamma}{2-\gamma} \frac{\epsilon}{2-\epsilon}\Leftrightarrow \beta_\epsilon > \frac{\epsilon}{\gamma}\bigg[\frac{2-\epsilon\gamma}{(2-\epsilon)(2-\gamma)(2-2\epsilon)}\bigg].
	\end{align*}
	Since $\epsilon \in (0,1/2)$ and $\gamma \in (0,1]$, we may choose $\beta_\epsilon=(2/\gamma)\cdot \epsilon$, along with any $\vartheta\geq 4$ and $\rho>0$ small enough to conclude.
\end{proof}

\subsection{Proof of Theorem \ref{thm:isotonic_reg}}

The proof of Theorem \ref{thm:isotonic_reg} follows from a more principled oracle inequality presented below---it captures the essential geometric property in the model that accounts for both the adaptation and robustness property of the shape-restricted LSE up to error distributions with an $L_{2,1}$ moment.  

\subsubsection{The general oracle inequality}

First some definitions.

\begin{definition}\label{def:l_2_shape_constraint}
	$\mathcal{F}$ is said to satisfy a \emph{convexity-based shape constraint} (under $P$) if 
	$\mathcal{F}$ is convex, and $\mathcal{F}(\delta)=\{f \in \mathcal{F}: Pf^2\leq \delta^2\}$ 
	admits a convex envelope  $F(\delta)$.
\end{definition}

\begin{definition}\label{def:basic_adaptive_set}
	$\mathcal{G}\subset \mathcal{F} $ is said to be a \emph{basic adaptive subset} of 
	$\mathcal{F}$ if $\mathcal{F}-\mathcal{G}\subset \mathcal{F}$. $\mathcal{G}_m$ is said 
	to be an \emph{$m$-th order adaptive subset} of $\mathcal{F}$ if for any $g_m \in \mathcal{G}_m$, 
	there is an interval partition $\{I_j\}_{j=1}^m$ of $\mathcal{X}=[0,1]$ and elements $\tilde{g}_j \in \mathcal{G}$ 
	such that $g_m=\sum_{i=1}^m \bm{1}_{I_j} \tilde{g}_j\in \mathcal{F}$.
\end{definition}

Before stating the general oracle inequality, recall that a function class $\mathcal{F}$ defined on 
$\mathcal{X}=[0,1]$ is called VC-major if the sets $\{x \in \mathcal{X}: f(x)\geq t\}$ with $f$ ranging over 
$\mathcal{F}$ and $t$ over $\R$ form a VC-class of sets.

\begin{theorem}\label{thm:l_2_shape_constraint_oracle_ineq}
	Consider the regression model (\ref{regression_model}) and the LSE $\hat{f}_n$ in (\ref{lse}). 
	Suppose that $\pnorm{f_0}{\infty}\vee \pnorm{f_0^\ast}{\infty}<\infty$, and that $\xi_1,\ldots,\xi_n$ 
	are mean zero errors independent of i.i.d.  covariates $X_i$'s with $\pnorm{\xi_1}{2,1}<\infty$. 
	Further assume that: 
(i) $\mathcal{F}$ satisfies a convexity-based shape constraint, and 
	$\mathcal{F}\cap L_\infty(B)$ is a VC-major class for any $B>0$, and 
(ii) $\pnorm{\hat{f}_n}{\infty}=\mathcal{O}_{\mathbf{P}}(1)$. Then for any $\delta \in (0,1)$, 
there exists $c\equiv c(\delta,\pnorm{\xi}{2,1},\mathcal{F},\pnorm{f_0}{\infty}, \pnorm{f_0^\ast}{\infty})>0$ 
such that with probability $1-\delta$,
	\begin{align*}
	\pnorm{\hat{f}_n-f_0^\ast}{L_2(P)}^2 \leq c \inf_{m \in \N}  
	\left(\inf_{f_m \in \mathcal{G}_m\cap L_\infty(\pnorm{f_0^\ast}{\infty})}\pnorm{f_m-f_0^\ast}{L_2(P)}^2+\frac{m  }{n} \cdot \log^{2} n\right),
	\end{align*}
	where $f_0^\ast = \argmin_{g \in \mathcal{F}\cap L_2(P)} \pnorm{f_0-g}{L_2(P)}$, and $\mathcal{G}_m$ is 
	an $m$-th order adaptive subset of $\mathcal{F}$.
\end{theorem}

The proof of Theorem \ref{thm:l_2_shape_constraint_oracle_ineq} will be deferred to the next 
subsection. We first use it to prove Theorem \ref{thm:isotonic_reg}. To this end, we only need to check: 
(i) the convexity-based shape constraint and VC-major condition of the isotonic and convex models; 
and (ii) the stochastic boundedness condition for the corresponding LSEs $\hat{f}_n$.

\begin{proof}[Proof of Theorem \ref{thm:isotonic_reg}]
	For the isotonic model $\mathcal{F}$, $\mathcal{F}$ is clearly convex, 
	and (\ref{eqn:envelope}) is an envelope for $\mathcal{F}(\delta)$ by the $L_2$ constraint 
	and monotonicity of the function class. Furthermore, it is clear by definition that 
	$\mathcal{F} \cap L_\infty(B)$ is VC-major. Similarly we can verify that the convex model 
	satisfies both the convexity-based shape constraint with the envelope (\ref{eqn:envelope}) 
	(cf. Lemma \ref{lem:convex_1dim_envelope}) and the VC-major condition.

	The stochastic boundedness of the isotonic and convex LSEs is established in the following lemma:
	\begin{lemma}\label{lem:uniform_bound_iso}
		If $\pnorm{f_0}{\infty}<\infty$ and $\pnorm{\xi_1}{2}<\infty$, then both the canonical isotonic and convex 
		LSEs are stochastically bounded:  $\pnorm{\hat{f}_n}{\infty}=\mathcal{O}_{\mathbf{P}}(1)$.
	\end{lemma}
	For the isotonic LSE, we use an explicit min-max representation (cf. \cite{robertson1988order}) 
	to prove this lemma, while for the convex LSE, the explicit characterization of the convex LSE derived in 
	\cite{groeneboom2001estimation} plays a crucial role. The details of the proof of this lemma can be 
	found in Section \ref{section:proof_remaining}. Now the claim of Theorem \ref{thm:isotonic_reg} follows 
	from Theorem \ref{thm:l_2_shape_constraint_oracle_ineq}, by noting that $\pnorm{f_0^\ast}{\infty}<\infty$ 
	under $\pnorm{f_0}{\infty}<\infty$, and that 
	$\inf_{f_m \in \mathcal{G}_m\cap L_\infty(\pnorm{f_0^\ast}{\infty})} \pnorm{f_m-f_0^\ast}{L_2(P)}^2
	 =\inf_{f_m \in \mathcal{G}_m} \pnorm{f_m-f_0^\ast}{L_2(P)}^2$ for isotonic model, 
	 and the same holds for the convex 
	model when $L_\infty(\pnorm{f_0^\ast}{\infty}$ is replaced by $L_\infty(C\pnorm{f_0^\ast}{\infty})$ for 
	some large enough $C>0$.
\end{proof}

\subsubsection{Proof of Theorem \ref{thm:l_2_shape_constraint_oracle_ineq}}

The first ingredient of the proof is the following proposition relating the convergence rate of 
$\hat{f}_n$ to the size of localized empirical processes.

\begin{proposition}\label{prop:lse}
	Consider the regression model (\ref{regression_model}) and the least squares estimator 
	$\hat{f}_n$ in (\ref{lse}). Suppose that $\xi_1,\ldots,\xi_n$ are mean-zero random variables independent of 
	$X_1,\ldots,X_n$, and $\mathcal{F}$ is convex with $\mathcal{F}-f_0^\ast\subset L_\infty(1)$.  Further assume that
	\begin{align}\label{cond:lse_generic}
	\E \sup_{f \in \mathcal{F}:\pnorm{f-f_0^\ast}{L_2(P)}\leq \delta} \biggabs{\frac{1}{\sqrt{n}}
	       \sum_{i=1}^n \xi_i(f-f_0^\ast)(X_i) }&\lesssim \phi_n(\delta),\nonumber\\
	\E \sup_{f \in \mathcal{F}:\pnorm{f-f_0^\ast}{L_2(P)}\leq \delta} \biggabs{\frac{1}{\sqrt{n}}
	       \sum_{i=1}^n \epsilon_i(f-f_0^\ast)(X_i) }&\lesssim \phi_n(\delta),\\
	\E \sup_{f \in \mathcal{F}:\pnorm{f-f_0^\ast}{L_2(P)}\leq \delta } \biggabs{\frac{1}{\sqrt{n}}
	       \sum_{i=1}^n \epsilon_i(f-f_0^\ast)(X_i)(f_0-f_0^\ast)(X_i) }&\lesssim \phi_n(\delta),\nonumber
	\end{align}
	hold for some $\phi_n$ such that $\delta\mapsto \phi_n(\delta)/\delta$ is non-increasing. 
	Then  \newline $\pnorm{\hat{f}_n-f_0^\ast}{L_2(P)}=\mathcal{O}_{\mathbf{P}}(\delta_n)$
	holds for any $\delta_n$ such that $\phi_n(\delta_n)\leq \sqrt{n}\delta_n^2$. 
\end{proposition}
\begin{proof}
This is a special case of Proposition \ref{prop:additive_lse}, the proof of which will be given therein.
\end{proof}

By Proposition \ref{prop:lse}, we only need to control the size of the empirical processes 
(\ref{cond:lse_generic}) centered at $f_0^\ast$. The following lemma will be useful in this regard by 
approximating $f_0^\ast$ via arbitrary $f_m \in \mathcal{G}_m$.

\begin{lemma}\label{lem:l_2_shape_constraint_mep_est}
	Suppose that the hypotheses of Theorem \ref{thm:l_2_shape_constraint_oracle_ineq} hold. 
	Let  $\{\delta_n\}_{n \in \N}$ be a sequence of positive real numbers such that $\delta_n\geq 1/n$. 
	Then for any $f_m \in \mathcal{G}_m\cap L_\infty(\pnorm{f_0^\ast}{\infty})$ and $B>0$,
	\begin{align*}
	&\max \bigg\{ \E \sup_{f\in \mathcal{F}: f-f_0^\ast \in L_2(\delta_n)\cap L_\infty(B) }\biggabs{\frac{1}{\sqrt{n}}\sum_{i=1}^{n} \xi_i (f-f_0^\ast)(X_i)},\\
	&\qquad\qquad \E \sup_{f \in \mathcal{F}: f-f_0^\ast \in L_2(\delta_n)\cap L_\infty(B) }\biggabs{\frac{1}{\sqrt{n}}\sum_{i=1}^{n} \epsilon_i (f-f_0^\ast)(X_i)},\\
	&\qquad\qquad\qquad \E \sup_{f \in \mathcal{F}: f-f_0^\ast \in L_2(\delta_n)\cap L_\infty(B) }\biggabs{\frac{1}{\sqrt{n}}\sum_{i=1}^{n} \epsilon_i (f-f_0^\ast)(X_i)(f_0-f_0^\ast)(X_i)} \bigg\}\\
	&\qquad \leq C_{\mathcal{F},\pnorm{f_0}{\infty}, \pnorm{f_0^\ast}{\infty},B}\cdot \pnorm{\xi_1}{2,1} \sqrt{\log (1/\delta_n)}\bar{L}_n \cdot (\delta_n\vee \pnorm{f_m-f_0^\ast}{L_2(P)})\sqrt{m},
	\end{align*}
where $\bar{L}_n\equiv \sqrt{\log n}$.
\end{lemma}

To prove Lemma \ref{lem:l_2_shape_constraint_mep_est}, we need the following form of a multiplier inequality proved in Proposition 1 of \cite{han2017sharp}.
\begin{lemma}\label{lem:interpolation_inequality}
	Suppose that $\xi_1,\ldots,\xi_n$ are i.i.d. mean-zero random variables independent of i.i.d. $X_1,\ldots,X_n$. Then for any function class $\mathcal{F}$,
	\begin{align}\label{ineq:interpolation_cont_prin}
	\E \bigg\lVert \sum_{i=1}^n \xi_i f(X_i)\bigg\lVert_{\mathcal{F}}\leq  \E \left[ \sum_{k=1}^n  (\abs{\eta_{(k)}}-\abs{\eta_{(k+1)}}) \E\bigg\lVert\sum_{i=1}^k \epsilon_i f(X_i)\bigg\lVert_{\mathcal{F}}\right]
	\end{align}
	where $\abs{\eta_{(1)}}\geq \cdots \geq \abs{\eta_{(n)}}\geq \abs{\eta_{(n+1)}}\equiv 0$ are the reversed order statistics for  $\{\abs{\xi_i-\xi_i'}\}_{i=1}^n$ with $\{\xi_i'\}$ being an independent copy of $\{\xi_i\}$.
\end{lemma}

	The following entropy estimate for bounded VC-major classes will be useful.
	\begin{lemma}\label{lem:subset_VC}
		Let $\mathcal{F}_0\subset L_\infty(1)$ be a VC-major class defined on $\mathcal{X}$. Then there exists some constant $C\equiv C_{\mathcal{F}_0}>0$ such that for any $\mathcal{F}\subset \mathcal{F}_0$, and any probability measure $Q$, the entropy estimate
		\begin{align*}
		 &\log \mathcal{N}\big(\epsilon \pnorm{F}{L_2(Q)}, \mathcal{F}, L_2(Q) \big) \leq \frac{C}{\epsilon} \log\bigg(\frac{C}{\epsilon}\bigg)  \log\bigg(\frac{1}{\epsilon \pnorm{F}{L_2(Q)}}\bigg),\textrm{ for all }\epsilon \in (0,1)
		\end{align*}
		holds for any envelope $F$ of $\mathcal{F}$.
	\end{lemma}

The proof of this lemma essentially follows from page 1171-1172 of \cite{gine2006concentration} 
with a minor modification. We include some details in Section \ref{section:proof_remaining} 
for the convenience of the reader.

We also need the following lemma concerning the envelope of a convex function given 
constraints on its $L_2$ size. The proof can be found in Lemma 7.3 of \cite{guntuboyina2013global}.
\begin{lemma}\label{lem:convex_1dim_envelope}
	If $f$ is a convex function on $[0,1]$ with $\int_0^1 \abs{f(x)}^2\ \d{x}\leq 1$, 
	then $ \abs{f(x)}\leq 2\sqrt{3}\big(x^{-1/2}\vee (1-x)^{-1/2}\big)$
	for all $x \in (0,1)$.
\end{lemma}

\begin{proof}[Proof of Lemma \ref{lem:l_2_shape_constraint_mep_est}]
	In the proof we omit the dependence on $L_\infty(B)$ if there is no confusion. 
	All three empirical processes can be handled in essentially the same way so we focus 
	on the most difficult first one (with $\xi_i$'s only admitting a $L_{2,1}$ moment). 
	We will apply Lemma \ref{lem:interpolation_inequality} in the following form:
	\begin{align}\label{ineq:iso_reg_0}
	&\E \sup_{f \in \mathcal{F}: f-f_0^\ast \in L_2(\delta_n) }\biggabs{\frac{1}{\sqrt{n}}\sum_{i=1}^{n} \xi_i (f-f_0^\ast)(X_i)}\\
	&\qquad \leq 3\pnorm{\xi_1}{2,1} \max_{1\leq k\leq n} 
	     \E \sup_{f \in \mathcal{F}: f-f_0^\ast \in L_2(\delta_n) }\biggabs{\frac{1}{\sqrt{k}}\sum_{i=1}^{k} \epsilon_i (f-f_0^\ast)(X_i)}.\nonumber
	\end{align}
	To see this, note that the right hand side of (\ref{ineq:interpolation_cont_prin}) can be bounded by
	\begin{align*}
	\E \bigg[ \sum_{k=1}^n  \sqrt{k} (\abs{\eta_{(k)}}-\abs{\eta_{(k+1)}}) \bigg]\cdot 
	      \max_{1\leq k\leq n} \E\bigg\lVert\frac{1}{\sqrt{k}}\sum_{i=1}^k \epsilon_i f(X_i)\bigg\lVert_{\mathcal{F}}
	\end{align*}
	where $
	\E \big[ \sum_{k=1}^n  \sqrt{k} (\abs{\eta_{(k)}}-\abs{\eta_{(k+1)}} ) \big] 
	\leq \sqrt{n} \pnorm{\eta_1}{2,1}\leq 3 \sqrt{n}\pnorm{\xi_1}{2,1}$. 
	The first inequality follows from similar lines as in the proof of Theorem 1 of 
	\cite{han2017sharp} and the second inequality uses Problem 2 on page 186 of \cite{van1996weak}. 
	This proves (\ref{ineq:iso_reg_0}).  Note that any $f_m \in \mathcal{G}_m$ has a representation 
	$f_m=\sum_{j=1}^m g_j \bm{1}_{I_j}$, where $\{I_j=[x_j,x_{j+1}]\}_{j=1}^m$ is a partition of $\mathcal{X}=[0,1]$ 
	with $x_1=0,x_{m+1}=1$ and $g_j \in \mathcal{G}$.  Then for any $f_m \in \mathcal{G}_m$, 
	the empirical process localized at $f_0^\ast$ can be controlled via
	\begin{align}\label{ineq:iso_reg_1}
	&\E \sup_{f\in \mathcal{F}: f-f_0^\ast \in L_2(\delta_n) \cap L_\infty(B)}\biggabs{\frac{1}{\sqrt{k}}\sum_{i=1}^{k} \epsilon_i (f-f_0^\ast)(X_i)}\\
	&\leq \E \sup_{ \substack{f\in \mathcal{F}: \pnorm{f-f_m}{L_2(P)}\leq \delta_n+\pnorm{f_m-f_0^\ast}{L_2(P)},\\ 
	       \pnorm{f}{\infty}\leq B+\pnorm{f_0^\ast}{\infty}} } \biggabs{\frac{1}{\sqrt{k}}
	           \sum_{i=1}^{k} \epsilon_i (f-f_m)(X_i)}+\pnorm{f_0^\ast-f_m}{L_2(P)},\nonumber
	\end{align}
	where the second term holds because the collection $\{f_0^\ast-f_m\}$ consists of just one element. 
	The first term in the above term can be further bounded by
	\begin{align}\label{ineq:iso_reg_2}
	&\sum_{j=1}^m \E \bigg[\frac{\sqrt{k_j}}{\sqrt{k}}\E \bigg[ \sup_{ \substack{f \in \mathcal{F}: \pnorm{f-f_m}{L_2(P)}
	     \leq \delta_n+\pnorm{f_m-f_0^\ast}{L_2(P)},\\ \pnorm{f}{\infty}\leq B+\pnorm{f_0^\ast}{\infty}} } 
	     \biggabs{\frac{1}{\sqrt{k_j}}\sum_{X_i \in I_j} \epsilon_i (f-g_j)(X_i) }\\
	&\qquad\qquad\qquad\qquad\qquad\qquad\qquad\qquad\qquad\qquad\qquad\qquad\bigg\lvert k_j(\bm{X})=k_j\bigg]\bigg]\nonumber\\
	&\leq \sum_{j=1}^m \E \bigg[\frac{\sqrt{k_j}}{\sqrt{k}}\E \bigg[\sup_{\substack{f|_{I_j} \in \mathcal{F}|_{I_j}:\\
			\pnorm{f}{\infty}\leq B+2\pnorm{f_0^\ast}{\infty}, \\  Pf^2
          \leq (\delta_n+\pnorm{f_m-f_0^\ast}{L_2(P)})^2 }}\biggabs{\frac{1}{\sqrt{k_j}}\sum_{X_i \in I_j} \epsilon_i f|_{I_j} (X_i) }
                  \bigg\lvert k_j(\bm{X})=k_j\bigg]\bigg]\nonumber
	\end{align}
	where $k_j(\bm{X})=\sum_{i=1}^k \bm{1}_{I_j}(X_i)$, and in the second line we used the definition of a basic adaptive subset (cf. Definition \ref{def:basic_adaptive_set}). From now on we write $\tilde{\delta}_n\equiv \delta_n + \pnorm{f_m-f_0^\ast}{L_2(P)}$ and 
	$B_0\equiv B+2\pnorm{f_0^\ast}{\infty}$ for notational convenience.  
	Since $\big(\mathcal{F}\cap L_\infty(B_0)\big)|_{I_j}$ is VC-major, so is its subset 
	$\mathcal{F}_{I_j}(\tilde{\delta}_n)\equiv \{ f|_{I_j} \in \big(\mathcal{F}\cap L_\infty(B_0)\big)|_{I_j}: Pf^2\leq \tilde{\delta}_n^2\}$. 
	It follows by Lemma \ref{lem:subset_VC} that there exists some $C\equiv C_{\mathcal{F},B_0}>0$ such that for any 
	probability measure $Q$ on $I_j$, and any $\epsilon \in (0,1)$,
	\begin{align*}
	 &\log \mathcal{N}\left(\epsilon \pnorm{F_{I_j}(\tilde{\delta}_n)}{L_2(Q)}, \mathcal{F}_{I_j}(\tilde{\delta}_n),L_2(Q)\right) 
	 \leq \frac{C}{\epsilon}  \log\bigg(\frac{C}{\epsilon}\bigg) \log\bigg(\frac{1}{\epsilon \pnorm{F_{I_j}(\tilde{\delta}_n)}{L_2(Q)}}\bigg),
	\end{align*}
    where $F_{I_j}(\delta)$ is any envelope for $\mathcal{F}_{I_j}(\delta)$. This enables us to apply the Koltchinskii-Pollard 
    maximal inequality to see that the summand (=conditional expectation) in the second line of (\ref{ineq:iso_reg_2}) 
    can be bounded by (further conditioning on which $X_i$'s lie in the interval $I_j$, each case corresponds to i.i.d. uniforms on $I_j$) 
	\begin{align}\label{ineq:iso_reg_4}
	&\int_0^1 \sqrt{  \frac{C}{\epsilon}  \log\bigg(\frac{C}{\epsilon}\bigg)  
	    \log\bigg(\frac{1}{\epsilon \inf_Q \pnorm{F_{I_j}(\tilde{\delta}_n) }{L_2(Q)}}\bigg)  }\ \d{\epsilon}\cdot  \sqrt{P_{I_j} F_{I_j}^2(\tilde{\delta}_n)},
	\end{align}
	where $P_{I_j}$ is the uniform distribution on $I_j$.
	
	In order to evaluate (\ref{ineq:iso_reg_4}), note that by the definition of convexity-based shape 
	constraint and Lemma \ref{lem:convex_1dim_envelope}, the envelopes $F_{I_j}(\delta)$'s can be taken 
	as the restrictions of the global envelope
	\begin{align*}
	F(\delta)(x)\equiv \left(\frac{\delta}{\sqrt{x}}\vee \frac{\delta}{\sqrt{1-x}}\right) \wedge B_0
	\end{align*}
	to the $I_j$'s.	Without loss of generality we assume: (i) $B_0=1$, (ii) $\tilde{\delta}_n^2 < 1/2$ 
	and (iii) $\tilde{\delta}_n^2$ and $1-\tilde{\delta}_n^2$ are one of the endpoints of some intervals in $\{I_j\}$ 
	(otherwise, we may take an alternative representation of $f_m \in \mathcal{G}_{m+2}$ by adding these two points). 
	
	Note that $\inf_Q \pnorm{F_{I_j}(\tilde{\delta}_n)}{L_2(Q)}\geq \sqrt{2}\tilde{\delta}_n> 1/n$ 
	by the assumption $\delta_n\geq 1/n$, and hence the integral term in (\ref{ineq:iso_reg_4}) can be bounded by
	\begin{align*}
	\int_0^1 \sqrt{ \frac{C}{\epsilon}\log\bigg(\frac{C}{\epsilon}\bigg) \log\bigg(\frac{n}{\epsilon}\bigg)}\ \d{\epsilon}  \lesssim \sqrt{\log n}\equiv \bar{L}_n.
    \end{align*}
    To handle the $\sqrt{P_{I_j} F_{I_j}^2(\tilde{\delta}_n)}$ term in (\ref{ineq:iso_reg_4}), define the index sets $
	\mathcal{J}_1\equiv \{ 1\leq j\leq m: I_j \subset [0,\tilde{\delta}_n^2]\},
	\mathcal{J}_2\equiv \{1\leq j\leq m: I_j \subset [\tilde{\delta}_n^2,1-\tilde{\delta}_n^2]\}$ and $
	\mathcal{J}_3\equiv \{1\leq j\leq m: I_j \subset [1-\tilde{\delta}_n^2,1]\}$. 
	It is easy to see that $\mathcal{J}_1\cup \mathcal{J}_2\cup \mathcal{J}_3=\{1,\ldots,m\}$. Clearly for $j \in \mathcal{J}_1\cup \mathcal{J}_3$,
	\begin{align*}
	P_{I_j} F_{I_j}^2 (\tilde{\delta}_n)  = \abs{I_j}^{-1} \int_{I_j} F_{I_j}^2(\tilde{\delta}_n)(x)\ \d{x} \leq 1,
	\end{align*}
	and for $j \in \mathcal{J}_2$, 
	\begin{align*}
	P_{I_j} F_{I_j}^2 (\tilde{\delta}_n)  &\leq \abs{I_j}^{-1}\tilde{\delta}_n^2 \int_{x_j}^{x_{j+1}} \left(\frac{1}{x}\vee \frac{1}{1-x}\right)\ \d{x}\\
	&\leq \abs{I_j}^{-1}\tilde{\delta}_n^2 \left[ \log \left(\frac{x_{j+1}}{x_j}\right)\vee \log \left(\frac{1-x_{j}}{1-x_{j+1}}\right) \right].
	\end{align*}
	Summarizing the above discussion shows that we can further bound (\ref{ineq:iso_reg_2}) by a $\mathcal{O}(\bar{L}_n)$ multiple of
	\begin{align}\label{ineq:iso_reg_3}
	&\sum_{j \in \mathcal{J}_1\cup \mathcal{J}_3} \E \left[ \sqrt{\frac{k_j}{k}}\cdot 1\right] + \sum_{j \in \mathcal{J}_2} \tilde{\delta}_n \cdot \E \left[ \sqrt{\frac{k_j}{k}}\cdot  \sqrt{\frac{\log(x_{j+1})-\log(x_j)}{x_{j+1}-x_j}}\right]\\
	&\qquad\qquad\qquad\qquad\qquad + \sum_{j \in \mathcal{J}_2} \tilde{\delta}_n \cdot \E \left[ \sqrt{\frac{k_j}{k}}\cdot  \sqrt{\frac{\log(1-x_{j})-\log(1-x_{j+1})}{(1-x_j)-(1-x_{j+1})}}\right]\nonumber\\
	&\equiv (I)+(II)+(III).\nonumber
	\end{align}
	The first term of (\ref{ineq:iso_reg_3}) is easy to handle: by the Cauchy-Schwarz inequality,
	\begin{align*}
	(I) \leq \sqrt{ k^{-1} \bigg(\E \sum_{j \in \mathcal{J}_1\cup \mathcal{J}_3} k_j(\bm{X})\bigg)\cdot \abs{\mathcal{J}_1\cup \mathcal{J}_3} } \leq \sqrt{\sum_{j \in \mathcal{J}_1\cup \mathcal{J}_3} \abs{I_j}  }\cdot \sqrt{m}\lesssim \tilde{\delta}_n \sqrt{m}.
	\end{align*}
	The second and third terms of (\ref{ineq:iso_reg_3}) can be handled in a similar fashion; we only consider the second term of (\ref{ineq:iso_reg_3}). Again by the Cauchy-Schwarz inequality, 
	\begin{align*}
	(II)&\leq \tilde{\delta}_n \sqrt{m} \cdot \sqrt{ \E \bigg[\sum_{j \in \mathcal{J}_2} \frac{k_j(\bm{X})}{k} \cdot \frac{\log(x_{j+1})-\log(x_j)}{x_{j+1}-x_j}\bigg]}\\
	& = \tilde{\delta}_n \sqrt{m}  \sqrt{\sum_{j \in \mathcal{J}_2} \big(\log(x_{j+1})-\log(x_j)\big) }\lesssim \sqrt{m}\cdot \tilde{\delta}_n \sqrt{\log (1/\tilde{\delta}_n)}.
	\end{align*}
	Collecting the above estimates, we see that (\ref{ineq:iso_reg_2}) can be bounded by a constant multiple of $\sqrt{m}\cdot \tilde{\delta}_n \sqrt{\log (1/\tilde{\delta}_n)}\bar{L}_n$. Thus, (\ref{ineq:iso_reg_1}) yields that 
	\begin{align*}
	\max_{1\leq k\leq n}\E \sup_{f \in \mathcal{F}: f-f_0^\ast \in L_2(\delta_n) }\biggabs{\frac{1}{\sqrt{k}}\sum_{i=1}^{k} \epsilon_i (f-f_0^\ast)(X_i)}\leq C'\sqrt{m}\cdot \tilde{\delta}_n \sqrt{\log (1/\tilde{\delta}_n)}\bar{L}_n.
	\end{align*}
	Combined with (\ref{ineq:iso_reg_0}), the claim of the lemma follows.
\end{proof}

\begin{proof}[Proof of Theorem \ref{thm:l_2_shape_constraint_oracle_ineq}]
	The proof follows easily from the reduction scheme Proposition \ref{prop:additive_lse} and Lemma \ref{lem:l_2_shape_constraint_mep_est}  by solving a quadratic inequality. We provide some details below. Abusing notation, we let $f_m\in \argmin_{g_m \in \mathcal{G}_m} \pnorm{g_m-f_0^\ast}{L_2(P)}$ and $m$ be the index attaining the infimum of the oracle inequality in the statement of the theorem. We only need to choose $\delta_n$ such that 
	\begin{align*}
	\sqrt{m}(\delta_n+\pnorm{f_m-f_0^\ast}{L_2(P)})\sqrt{\log (1/\delta_n)} \bar{L}_n\leq c_{\delta,\mathcal{F},\pnorm{f_0^\ast}{\infty},\pnorm{\xi}{2,1}} \sqrt{n}\delta_n^2.
	\end{align*}
	Suppose $\log(1/\delta_n)\lesssim \log n$. Then we can easily solve for the zeros for quadratic forms to see that the inequality in the last display holds if
	\begin{align*}
	\delta_n^2 \gtrsim \frac{m\bar{L}_n^2 \log n} {n}+\sqrt{\frac{m\bar{L}_n^2\log n}{n}}\pnorm{f_m-f_0^\ast}{L_2(P)}.
	\end{align*}
	The assumption $\log(1/\delta_n)\lesssim \log n$ apparently holds. The right hand side of the above display can be further bounded up to a constant by $
	\frac{m\bar{L}_n^2 \log n}{n}+\pnorm{f_m-f_0^\ast}{L_2(P)}^2$ 
	by the basic inequality $ab\leq (a^2+b^2)/2$, thereby completing the proof of Theorem \ref{thm:l_2_shape_constraint_oracle_ineq}.
\end{proof}

\subsection{Proof of Theorem \ref{thm:additive_isotonic_reg}}

The proof of Theorem \ref{thm:additive_isotonic_reg} follows a similar strategy as that of Theorem \ref{thm:l_2_shape_constraint_oracle_ineq}. First we need the following reduction scheme.

\begin{proposition}\label{prop:additive_lse}
	Consider the additive model (\ref{additive_model}) and the least squares estimator $\hat{f}_n$ in (\ref{lse_additive_model}).  Suppose that $\xi_1,\ldots,\xi_n$ are mean-zero random variables independent of $(X_1,Z_1),\ldots,(X_n,Z_n)$, and $\mathcal{F}$ is convex with $\mathcal{F}-f_0^\ast\subset L_\infty(1)$.  Further assume that all three parts of (\ref{cond:lse_generic})
	and
	\begin{align}\label{cond:ep_add_2}
	\E \sup_{ \substack{f \in \mathcal{F}:\pnorm{f-f_0^\ast}{L_2(P)}\leq \delta \\ h\in \mathcal{H}}} \biggabs{\frac{1}{\sqrt{n}}\sum_{i=1}^n \epsilon_i(f-f_0^\ast)(X_i)(h-(\phi_0-f_0))(X_i, Z_i) }&\lesssim \phi_n(\delta),
	\end{align}
	hold for some $\phi_n$ such that $\delta\mapsto \phi_n(\delta)/\delta$ is non-increasing. 
	Then \newline
	 $\pnorm{\hat{f}_n-f_0^\ast}{L_2(P)}=\mathcal{O}_{\mathbf{P}}(\delta_n)$ 
	holds for any $\delta_n$ such that $\phi_n(\delta_n)\leq \sqrt{n}\delta_n^2$. 
\end{proposition}

\begin{proof}
 Recall that $f_0 = P_Z \phi_0(\cdot, Z)$. By the definition of the LSE,
	\begin{align*}
	&\Prob_n(\phi_0+\xi-\hat{f}_n-\hat{h}_n)^2\leq \Prob_n(\phi_0+\xi-f_0^\ast-\hat{h}_n)^2\\
	&\Leftrightarrow \quad \Prob_n(f_0^\ast-\hat{f}_n)\big(2\phi_0+2\xi-\hat{f}_n-f_0^\ast-2\hat{h}_n\big)\leq 0\\
	&\Leftrightarrow \quad \Prob_n(f_0^\ast-\hat{f}_n)^2+2\Prob_n(f_0^\ast-\hat{f}_n)\big(\phi_0+\xi-f_0^\ast-\hat{h}_n\big)\leq 0\\
	&\Leftrightarrow \quad -\Prob_n(f_0^\ast-\hat{f}_n)^2-2\Prob_n(f_0^\ast-\hat{f}_n)\xi -2\Prob_n(f_0^\ast-\hat{f}_n)(f_0-f_0^\ast)\\
	&\qquad\qquad\qquad - 2\Prob_n(f_0^\ast-\hat{f}_n)(\phi_0-f_0-\hat{h}_n)\geq 0\\
	&\Leftrightarrow \quad -(\Prob_n-P)\bigg[(f_0^\ast-\hat{f}_n)^2-2\xi(f_0^\ast-\hat{f}_n)\bigg]- P(f_0^\ast-\hat{f}_n)^2\\
	&\qquad\qquad\qquad -2(\Prob_n-P)(f_0^\ast-\hat{f}_n)(f_0-f_0^\ast)-2P(f_0^\ast-\hat{f}_n)(f_0-f_0^\ast)\\
	&\qquad\qquad\qquad\qquad-2(\Prob_n-P)(f_0^\ast-\hat{f}_n)(\phi_0-f_0-\hat{h}_n)\geq 0.
	\end{align*}
The last equivalence holds since
\begin{align*}
&P(f_0^\ast-\hat{f}_n)(X)(\phi_0-f_0-\hat{h}_n)(X,Z) \\
&= P \bigg[(f_0^\ast-\hat{f}_n)(X) P \big[(\phi_0-f_0-\hat{h}_n)(X,Z)\big\lvert X\big]\bigg]\\
& = P \bigg[(f_0^\ast-\hat{f}_n)(X) \big(P \big[\phi_0(X,Z)\big\lvert X\big]- f_0(X)-P \hat{h}_n(Z)\big)\bigg] = 0,
\end{align*}
where we used (i) $P [\phi_0(X,Z)|X]=f_0(X)$, and (ii) $P h = 0$ for all $h \in \mathcal{H}$. Now since $f_0^\ast \in \argmin_{g \in \mathcal{F}\cap L_2(P)} \pnorm{f_0-g}{L_2(P)}$, it follows from the convexity of $\mathcal{F}$ that $P(f_0^\ast-\hat{f}_n)(f_0-f_0^\ast)\geq 0$ [more specifically, for each $\epsilon>0$, since $(1-\epsilon)f_0^\ast+\epsilon \hat{f}_n^\ast \in \mathcal{F}\cap L_2(P)$ by convexity of $\mathcal{F}$, the definition of $f_0^\ast$ yields that $P(f_0-f_0^\ast)^2\leq P(f_0-(1-\epsilon)f_0^\ast-\epsilon \hat{f}_n)^2=P(f_0-f_0^\ast+\epsilon(f_0^\ast-\hat{f}_n))^2$. The claim follows by expanding the square and taking $\epsilon \to 0$].  This implies that, with $S_j(\delta_n)\equiv \{ f \in \mathcal{F}: 2^{j-1}\delta_n < \pnorm{f-f_0^\ast}{L_2(P)}\leq 2^j \delta_n\}$, on the event $\{2^{j-1}\delta_n<\pnorm{\hat{f}_n-f_0^\ast}{L_2(P)}\leq 2^j \delta_n\}$, it holds that
\begin{align*}
& \sup_{f \in S_j(\delta_n)} \abs{(\Prob_n-P)(f-f_0^\ast)^2}+2\sup_{f \in S_j(\delta_n)}\abs{(\Prob_n-P)\xi(f-f_0^\ast)}\\
& \qquad+2\sup_{f \in S_j(\delta_n)} \abs{(\Prob_n-P)(f-f_0^\ast)(f_0-f_0^\ast)}\\
&\qquad\qquad +2\sup_{f \in S_j(\delta_n), h \in \mathcal{H}}\abs{(\Prob_n-P)(f-f_0^\ast)(h-(\phi_0-f_0))}\\
&\geq -(\Prob_n-P)\bigg[(f_0^\ast-\hat{f}_n)^2-2\xi(f_0^\ast-\hat{f}_n)\bigg]\\
&\qquad -2(\Prob_n-P)(f_0^\ast-\hat{f}_n)(f_0-f_0^\ast)-2(\Prob_n-P)(f_0^\ast-\hat{f}_n)(\phi_0-f_0-\hat{h}_n)\\
&\geq 2^{2j-2} \delta_n^2.
\end{align*}
Hence by symmetrization, the contraction principle for Rademacher processes and the assumptions we see that 
\begin{align*}
&\Prob \big(\pnorm{\hat{f}_n-f_0^\ast}{L_2(P)}>2^{M-1} \delta_n\big)\\
&\leq \sum_{j\geq M} \Prob\bigg(\sup_{f \in S_j(\delta_n)} \abs{(\Prob_n-P)(f-f_0^\ast)^2}+2\sup_{f \in S_j(\delta_n)}\abs{(\Prob_n-P)\xi(f-f_0^\ast)}\\
& \qquad\qquad+2\sup_{f \in S_j(\delta_n)} \abs{(\Prob_n-P)(f-f_0^\ast)(f_0-f_0^\ast)}\\
&\qquad\qquad\qquad +2\sup_{f \in S_j(\delta_n), h \in \mathcal{H}}\abs{(\Prob_n-P)(f-f_0^\ast)(h-(\phi_0-f_0))}\geq 2^{2j-2}\delta_n^2\bigg)\\
&\lesssim \sum_{j\geq M} \big(2^{2j}\sqrt{n}\delta_n^2 \big)^{-1}\bigg( \E \pnorm{\G_n}{\mathcal{F}_0(2^j\delta_n)} \vee  \E \pnorm{\G_n}{ \mathcal{F}_0(2^j\delta_n)\otimes \xi }  \\
&\qquad\qquad \qquad\qquad \qquad\vee \E \pnorm{\G_n}{ \mathcal{F}_0(2^j\delta_n)\otimes (f_0-f_0^\ast) }   \vee \E \pnorm{\G_n}{ \mathcal{F}_0(2^j\delta_n)\otimes (\mathcal{H}-(\phi_0-f_0))} \bigg)\\
& \leq C\sum_{j\geq M} \frac{\phi_n(2^j \delta_n)}{2^{2j} \sqrt{n} \delta_n^2}\leq C\sum_{j\geq M} \frac{\phi_n( \delta_n)}{2^{j} \sqrt{n} \delta_n^2}\lesssim \sum_{j\geq M} 2^{-j}\to 0
\end{align*}
as $M \to \infty$. Here we denote $\mathcal{F}_0\equiv \mathcal{F}-f_0^\ast$, and in the last sequence of inequalities we used the assumption that $\delta \mapsto \phi_n(\delta)/\delta$ is non-decreasing and the definition of $\delta_n$. This completes the proof.
\end{proof}

By Proposition \ref{prop:additive_lse}, apart from the empirical processes in Lemma \ref{lem:l_2_shape_constraint_mep_est}, we also need to control the empirical process (\ref{cond:ep_add_2}) indexed by a suitably localized subset of $\mathcal{F}\otimes (\mathcal{H}-(\phi_0-f_0))\equiv \{f(x)\big(h(z)-\phi_0(x,z)-f_0(x)\big): f \in \mathcal{F}, h \in \mathcal{H}\}$. In a related work, \cite{van2014uniform} derived bounds for similar empirical processes under $L_\infty$-type entropy conditions for both $\mathcal{F}$ and $\mathcal{H}$ (cf. Theorem 3.1 of \cite{van2014uniform}), which apparently fail for shape constrained classes.

\begin{lemma}\label{lem:additive_mep_est}
	Suppose that the hypotheses of Theorem \ref{thm:additive_isotonic_reg} hold. Let  $\{\delta_n\}_{n \in \N}$ be a sequence of positive real numbers such that $\delta_n\geq 1/n$. Then for any $f_m \in \mathcal{G}_m\cap L_\infty(\pnorm{f_0^\ast}{\infty})$, and $B>0$,
	\begin{align*}
	&\E \sup_{ \substack{f \in \mathcal{F}: f-f_0^\ast \in L_2(\delta_n)\cap L_\infty(B)\\ h \in \mathcal{H}} }\biggabs{\frac{1}{\sqrt{n}}\sum_{i=1}^{n} \epsilon_i (f-f_0^\ast)(X_i)(h-h_0)(X_i, Z_i)}\\
	&\qquad \leq C_{ \mathcal{H},\mathcal{F},\pnorm{\phi_0}{\infty}, \pnorm{f_0^\ast}{\infty},B}\cdot  \sqrt{\log (1/\delta_n)} \bar{L}_n\cdot (\delta_n\vee \pnorm{f_m-f_0^\ast}{L_2(P)})\sqrt{m}.
	\end{align*}
	Here $\bar{L}_n\equiv \sqrt{\log n}$.
\end{lemma}

We need some technical lemmas. Recall $P_X,P_Z$ are the marginal probability distributions of $(X,Z)$, i.e. uniform distribution on $[0,1]$.

\begin{lemma}\label{lem:entropy_ind}
Let $\mathcal{H}$ be a class of measurable functions defined on $[0,1]$, and let $f \in L_2(P_X), g \in L_2(P)$. Then for any probability measure $Q$ on $[0,1]^2$,
\begin{align*}
\mathcal{N}\big(\epsilon \pnorm{f\otimes 1}{L_2(Q)}, f\otimes \big(\mathcal{H}-g\big), L_2(Q)\big)\leq \mathcal{N}_{}\big(\epsilon, \mathcal{H}, L_\infty\big).
\end{align*}
\end{lemma}

\begin{lemma}\label{lem:cond_H2_1}
	Suppose the conditions on $\mathcal{H}$ in Theorem \ref{thm:additive_isotonic_reg} hold and $\mathcal{F}$ is the class of monotonic non-decreasing or convex functions on $[0,1]$. Then for any $\mathcal{F}'\subset \mathcal{F}\cap L_\infty(1)$ and any probability measure $Q$ on $[0,1]^2$, the entropy estimate
	\begin{align*}
	& \log \mathcal{N}\big(\epsilon \pnorm{F'\otimes 1}{L_2(Q)}, \mathcal{F}'\otimes (\mathcal{H}-(\phi_0-f_0)), L_2(Q) \big)\\
	&\qquad\lesssim  \frac{1}{\epsilon}\log \bigg(\frac{1}{\epsilon}\bigg) \log\bigg(\frac{1}{\epsilon\pnorm{F'\otimes 1}{L_2(Q)} }\bigg) \vee \epsilon^{-\gamma},\textrm{ for all }\epsilon \in (0,1)
	\end{align*}
	holds for any envelope $F'$ of $\mathcal{F}'$. The constant in the above estimate does not depend on the choice of $\mathcal{F}'$ or $Q$.
\end{lemma}

The proofs of Lemmas \ref{lem:entropy_ind} and \ref{lem:cond_H2_1} are standard. We include the details in Section \ref{section:proof_remaining} for completeness.

\begin{proof}[Proof of Lemma \ref{lem:additive_mep_est}]
The proof follows the same strategy as that of Lemma \ref{lem:l_2_shape_constraint_mep_est}. We only prove the isotonic case $\mathcal{G}_m=\mathcal{M}_m$; the convex case follows by similar arguments. As in the proof of Lemma \ref{lem:l_2_shape_constraint_mep_est}, we will omit the explicit dependence on $L_\infty(B)$ if no confusion arises. Note that 
\begin{align}\label{ineq:additive_iso_1}
&\E \sup_{f \in \mathcal{F}: f-f_0^\ast \in L_2(\delta_n), h \in \mathcal{H}}\biggabs{\frac{1}{\sqrt{n}}\sum_{i=1}^{n} \epsilon_i (f-f_0^\ast)(X_i)(h-(\phi_0-f_0))(X_i, Z_i)}\\
&\leq \E \sup_{ \substack{ f \in \mathcal{F}: \pnorm{f-f_m}{L_2(P)} \leq \delta_n+\pnorm{f_m-f_0^\ast}{L_2(P)}, \\ h \in \mathcal{H}   }} \biggabs{\frac{1}{\sqrt{n}}\sum_{i=1}^{n} \epsilon_i (f-f_m)(X_i)(h-(\phi_0-f_0))(X_i, Z_i)}\nonumber\\
&\qquad\qquad+ \E \sup_{h \in \mathcal{H}} \biggabs{\frac{1}{\sqrt{n}}\sum_{i=1}^{n} \epsilon_i (f_m-f_0^\ast)(X_i)(h-(\phi_0-f_0))(X_i, Z_i)}\equiv (I)+(II).\nonumber
\end{align}
We first handle $(II)$ in (\ref{ineq:additive_iso_1}). The entropy assumption on $\mathcal{H}$ coupled with Lemma \ref{lem:entropy_ind} entails that the uniform entropy integral for the class $(f_m-f_0^\ast)\otimes \big(\mathcal{H}-(\phi_0-f_0)\big)$ converges. By Theorem 2.14.1 of \cite{van1996weak}, we have the following estimate
\begin{align*}
(II) &\leq C_{\mathcal{H}}\pnorm{f_m-f_0^\ast}{L_2(P)}.
\end{align*}
For the first term $(I)$ in (\ref{ineq:additive_iso_1}), we mimic the proof strategy in Lemma \ref{lem:l_2_shape_constraint_mep_est}: any piecewise constant $f_m \in \mathcal{M}_m$ has a representation $f_m=\sum_{j=1}^m g_j \bm{1}_{I_j}$, where $\{I_j=[x_j,x_{j+1}]\}_{j=1}^m$ is a partition of $[0,1]$ with $x_1=0,x_{m+1}=1$ and $g_j$ takes constant values on the intervals $I_j$.  Then for such $f_m \in \mathcal{M}_m$, write $\tilde{I}_j = I_j\times [0,1]$, we have
\begin{align*}
&\sum_{j=1}^m \E \bigg[\frac{\sqrt{n_j}}{\sqrt{n}}\E \bigg[ \sup_{ \substack{f \in \mathcal{F}: f-f_m \in L_2(\tilde{\delta}_n)\\ h \in \mathcal{H}} } \biggabs{\frac{1}{\sqrt{n_j}}\sum_{(X_i,Z_i) \in \tilde{I}_j} \epsilon_i (f-g_j)(X_i)(h-(\phi_0-f_0))(X_i,Z_i) }\\
&\qquad\qquad\qquad\qquad\qquad\qquad\qquad\qquad\qquad\qquad\qquad\qquad\bigg\lvert n_j(\bm{X},\bm{Z})=n_j\bigg]\bigg]\\
&\leq \sum_{j=1}^m \E \bigg[\frac{\sqrt{n_j}}{\sqrt{n}}\E \bigg[\sup_{\substack{f|_{I_j} \in \mathcal{F}|_{I_j}:\\ \pnorm{f}{\infty}\leq B+2\pnorm{f_0^\ast}{\infty}, \\  P_X f^2\leq \tilde{\delta}_n^2, h \in \mathcal{H} }}\biggabs{\frac{1}{\sqrt{n_j}}\sum_{(X_i,Z_I) \in \tilde{I}_j} \epsilon_i f\otimes (h-(\phi_0-f_0))|_{\tilde{I}_j} (X_i,Z_i) }\\
&\qquad\qquad\qquad\qquad\qquad\qquad\qquad\qquad\qquad\qquad\qquad\qquad \bigg\lvert n_j(\bm{X},\bm{Z})=n_j\bigg]\bigg]
\end{align*}
where $\tilde{\delta}_n\equiv \delta_n+\pnorm{f_m-f_0^\ast}{L_2(P)}$. Here $n_j(\bm{X},\bm{Z})=\sum_{i=1}^n \bm{1}_{\tilde{I}_j}(X_i,Z_i)$ and in the second line we used the fact that $(f-f_m)|_{I_j} \in \mathcal{F}|_{I_j}$. By Lemma \ref{lem:cond_H2_1} and the Koltchinskii-Pollard maximal inequality, each summand of the above display can be bounded up to a constant (depending on $\mathcal{F}, \mathcal{H}, \pnorm{\phi_0}{\infty}$) by 
\begin{align*}
&\int_0^1\sqrt{\frac{1}{\epsilon}  \log \bigg(\frac{1}{\epsilon}\bigg) \log\bigg(\frac{1}{\epsilon\pnorm{F_{I_j}(\tilde{\delta}_n)\otimes 1}{L_2(Q)} }\bigg)\vee \epsilon^{-\gamma}}\ \d{\epsilon} \\
&\qquad\qquad \times  \left(P_{\tilde{I}_j} F_{I_j}^2(\tilde{\delta}_n) \right)^{1/2}\lesssim \bar{L}_n\cdot  \sqrt{P_{I_j}F^2_{I_j}(\tilde{\delta}_n)},
\end{align*}
 where $P_{\tilde{I}_j}$ is the uniform distribution on $\tilde{I}_j$ and $F_{I_j}(\delta)$ is the envelope for $\big(\mathcal{F}\cap L_\infty(B+2\pnorm{f_0^\ast}{\infty})\cap L_2(\delta)\big)|_{I_j}$, and the inequality in the above display follows from similar arguments as in the proof of Lemma \ref{lem:l_2_shape_constraint_mep_est}. From here the proof proceeds along the same lines as that of the proof for Lemma \ref{lem:l_2_shape_constraint_mep_est}. 
\end{proof}

\begin{proof}[Proof of Theorem \ref{thm:additive_isotonic_reg}]
The proof of Theorem \ref{thm:additive_isotonic_reg} follows the arguments of the proof of Theorem \ref{thm:l_2_shape_constraint_oracle_ineq} by using Proposition \ref{prop:additive_lse} along with Lemmas \ref{lem:l_2_shape_constraint_mep_est} and \ref{lem:additive_mep_est}, combined with the stochastic boundedness of the LSE:
\begin{lemma}\label{lem:uniform_bound_additive_iso}
	Suppose that the hypotheses of Theorem \ref{thm:additive_isotonic_reg} hold (except that $\mathcal{H}$ is only required to have a continuously square integrable envelope $P_Z H^2<\infty$). Then both the canonical isotonic and convex LSEs in the additive regression model (\ref{lse_additive_model}) are stochastically bounded:  $\pnorm{\hat{f}_n}{\infty}=\mathcal{O}_{\mathbf{P}}(1)$.
\end{lemma}
The proof of this lemma will be detailed in Section \ref{section:proof_remaining}, and hence completes the proof of Theorem \ref{thm:additive_isotonic_reg}.
\end{proof}

\section{Proofs of technical results}\label{section:proof_remaining}

In this section, we collect the proofs for technical results in three groups:
\begin{enumerate}
	\item the key Proposition \ref{prop:lower_bound_envelope} used in the proof of Theorem \ref{thm:envelope_rate_lower_bound};
	\item entropy results in Lemmas \ref{lem:subset_VC}, \ref{lem:entropy_ind} and \ref{lem:cond_H2_1};
	\item stochastic boundedness for shape-restricted LSEs in Lemmas \ref{lem:uniform_bound_iso} and \ref{lem:uniform_bound_additive_iso}. 
\end{enumerate}

\subsection{Proof of Proposition \ref{prop:lower_bound_envelope}}

In the next few subsections, we will prove Proposition \ref{prop:lower_bound_envelope} step by step.

\subsubsection{Construction of $\tilde{\mathcal{F}}$}

First consider the case $\gamma \in (0,1)$. We will do the construction iteratively. For $l=1$, since $[0,1]$ contains $\floor{2^{\frac{1}{1-\gamma}} }$ many equal-length intervals (with length $\big(2^{\frac{1}{1-\gamma}} \big)^{-1}$), we can pick $2$ intervals among them; this is denoted $\tilde{\mathscr{C}}_1$. For $l=2$, each interval in $\tilde{\mathscr{C}}_1$ contains $\floor{2^{\frac{1}{1-\gamma}} }$ many equal-length subintervals with length $\big(2^{\frac{1}{1-\gamma}} \big)^{-2}$, we can pick $2$ subintervals among each of the interval; this is denoted $\tilde{\mathscr{C}}_2$. In this way we can define iteratively $\tilde{\mathscr{C}}_l$ for any $l \in \N$. Let $\tilde{\mathcal{F}}_l\equiv \{\bm{1}_{I}:I \in \tilde{\mathscr{C}}_l\}$. Clearly $\abs{\tilde{\mathcal{F}}_l}=2^l$ and contains indicators over intervals in $[0,1]$ with length $\big(2^{\frac{1}{1-\gamma}} \big)^{-l}$. Now let $\tilde{\mathcal{F}}\equiv \cup_{l \in \N} \tilde{\mathcal{F}}_l \cup \{\bm{0}\}$ where $\bm{0}$ denotes a mapping taking identical value $0$. Next, for $\gamma=1$, let $\tilde{\mathcal{F}}\equiv \{\bm{1}_{[0,\delta]}:0\leq \delta\leq 1\}$.

We show that the constructed $\tilde{\mathcal{F}}$ satisfies the desired growth condition (\ref{cond:size_envelope}). Recall $P$ is the uniform distribution on $[0,1]$.

\begin{lemma}\label{lem:tilde_F_cond}
	It holds that
	\begin{align*}
	\pnorm{ \tilde{F}(\delta)}{L_2(P)} \leq \sqrt{2} \delta^\gamma,
	\end{align*}
	where $\tilde{F}(\delta)$ denotes the envelope for $\tilde{\mathcal{F}}(\delta)$. 
\end{lemma}

\begin{proof}
	The claim is trivial for $\gamma=1$. For $\gamma \in (0,1)$, since each element in  $\tilde{\mathscr{C}}_{l+1}$ is contained in some element in $\tilde{\mathscr{C}}_{l}$, we only need to count the number of intervals for the smallest level $l(\delta)$ such that the length of intervals in $\tilde{\mathcal{F}}_{l(\delta)}$ is no more than $\delta^2$. In other words, $l(\delta)$ is the integer for which
	\begin{align*}
	\big(2^{\frac{1}{1-\gamma}} \big)^{-l(\delta)}\leq \delta^2,\quad \big(2^{\frac{1}{1-\gamma}} \big)^{-l(\delta)+1}> \delta^2.
	\end{align*}
	Hence the number of intervals in $\tilde{\mathcal{F}}_{l(\delta)}$ is $N(\delta)=2^{l(\delta)} \in [\delta^{-(2-2\gamma)},2\delta^{-(2-2\gamma)}]$, from which the claim of the lemma holds.
\end{proof}

\subsubsection{Proof of claim (1) of Proposition \ref{prop:lower_bound_envelope}}

The following standard Paley-Zygmund lower bound will be used.
\begin{lemma}[Paley-Zygmund]\label{lem:paley_zygmund}
	Let $Z$ be any non-negative random variable. Then for any $\epsilon>0$,  $
	\Prob (Z>\epsilon \E Z) \geq \left(\frac{(1-\epsilon) \E Z}{(\E Z^q)^{1/q}}\right)^{q'}$, 
	where $q,q' \in (1,\infty)$ are conjugate indices: $1/q+1/q'=1$.
\end{lemma}

We need the following exact characterization concerning the size of maxima of a sequence of independent random variables due to \cite{gine1983central}, see also Corollary 1.4.2 of \cite{de2012decoupling}.

\begin{lemma}\label{lem:characterization_maxima}
	Let $\xi_1,\ldots,\xi_n$ be a sequence of independent non-negative random variables such that $\pnorm{\xi_i}{r}<\infty$ for all $1\leq i\leq n$. For $\lambda>0$, set $
	\delta_0(\lambda)\equiv \inf\left\{t>0: \sum_{i=1}^n \Prob(\xi_i>t)\leq \lambda\right\}$.
	Then
	\begin{align*}
	\frac{1}{1+\lambda}\sum_{i=1}^n \E \xi_i^r\bm{1}_{\xi_i>\delta_0}\leq \E\max_{1\leq i\leq n} \xi_i^r\leq \frac{1}{1\wedge \lambda}\sum_{i=1}^n \E \xi_i^r\bm{1}_{\xi_i>\delta_0}.
	\end{align*}
\end{lemma}

\begin{proof}[Proof of Proposition \ref{prop:lower_bound_envelope}, claim (1)]
	\textbf{(Case 1: $0< \gamma <1$).} Recall $\delta_2 \equiv \vartheta n^{-\frac{1}{2(2-\gamma)}}$. Then by the proof of Lemma \ref{lem:tilde_F_cond}, we see that there exists some level $l(\delta_2) \in \N$ such that the $N(\delta_2)$ many intervals $\{I_l\}_{l=1}^{N(\delta_2)}$ in $\tilde{\mathcal{F}}_{l(\delta_2)}$ have length at most $\delta_2^2$ and at least $2^{-1/(1-\gamma)}\delta_2^2$, while the number of intervals satisfies $\vartheta^{-(2-2\gamma)} n^{\frac{1-\gamma}{2-\gamma}}\leq N(\delta_2)\leq 2\vartheta^{-(2-2\gamma)} n^{\frac{1-\gamma}{2-\gamma}}$.  Let $\mathcal{E}_n$ be the event that all intervals $\{I_l\}_{l=1}^{N(\delta_2)}$ contain at least $2^{-\frac{2-\gamma}{1-\gamma}}\vartheta^2 n^{\frac{1-\gamma}{2-\gamma}}$ of the $X_i$'s and at most $\frac{5}{4}\vartheta^2 n^{\frac{1-\gamma}{2-\gamma}}$ of the $X_i$'s. Then by a union bound and Bernstein's inequality (cf. (2.10) of \cite{boucheron2013concentration}),
	\begin{align}
	\label{ineq:lower_bound_envelope_1}
	\Prob(\mathcal{E}_n^c)&\leq \Prob\bigg(\max_{1\leq l\leq N(\delta_2)} \bigg\lvert\sum_{i=1}^n \bm{1}_{I_l}(X_i)-n\abs{I_l} \bigg\rvert>2^{-\frac{2-\gamma}{1-\gamma}}\vartheta^2 n^{\frac{1-\gamma}{2-\gamma}}  \bigg)\\
	&\leq 2\vartheta^{-(2-2\gamma)} n^{\frac{1-\gamma}{2-\gamma}}\exp\big(-c_\gamma \vartheta^2 n^{\frac{1-\gamma}{2-\gamma}} \big).\nonumber
	\end{align}
	Let $\mathcal{I}_l\equiv \{X_i \in I_l\}$ for $1\leq l\leq N(\delta_2)$ and $\{\xi_i^{(l)}\}_{i,l\geq 1}$ be i.i.d. random variables with the same law as $\xi_1$. Then for some $t_n>0$ to be determined later,
	\begin{align}
	\label{ineq:lower_bound_envelope_2}
	\Prob\bigg( \sup_{f \in \tilde{\mathcal{F}}: Pf^2\leq \delta_2^2} \bigg\lvert \sum_{i=1}^n \xi_i f(X_i)\bigg\lvert\geq t_n\bigg)
	&\geq \E_{\bm{X}} \bigg[ \Prob_{\bm{\xi}} \bigg(\max_{1\leq l\leq N(\delta_2)}  \bigg\lvert \sum_{i=1}^n \xi_i \bm{1}_{I_l} (X_i)\bigg\rvert\geq t_n \bigg)\bm{1}_{\mathcal{E}_n}\bigg]\\
	&= \E_{\bm{X}} \bigg[ \Prob_{\bm{\xi}} \bigg(\max_{1\leq l\leq N(\delta_2)}  \bigg\lvert\sum_{i=1}^{ \abs{\mathcal{I}_l} } \xi_i^{(l)} \bigg\rvert\geq t_n \bigg)\bm{1}_{\mathcal{E}_n}\bigg].\nonumber
	\end{align}
	Our goal now is to make a good choice of the law for $\xi_\cdot^{(\cdot)}$'s so that we may obtain a good estimate for $t_n$ and thereby using the Paley-Zygmund argument. Let $\xi_1$ be distributed according to the symmetric $\alpha_\epsilon\equiv 2-\epsilon$ stable law, i.e. the characteristic function of $\xi_1$ is $\varphi_{\xi_1}(t)=\exp(-\abs{t}^{\alpha_\epsilon})$.
	Apparently, $k^{-1/\alpha_\epsilon}\sum_{i=1}^k \xi_i^{(l)}$ has the same law as that of $\xi_1$, and hence we can take
	\begin{align}\label{ineq:lower_bound_envelope_3}
	t_n= \frac{1}{2} \bigg(2^{-\frac{2-\gamma}{1-\gamma}}\vartheta^2 n^{\frac{1-\gamma}{2-\gamma}}\bigg)^{1/\alpha_\epsilon}\E_{\bm{\xi}} \max_{1\leq l\leq N(\delta_2)}  \abs{\xi_l}.
	\end{align}
	Then the conditional probability in the last line of (\ref{ineq:lower_bound_envelope_2}) can be bounded from below by
	\begin{align}\label{ineq:lower_bound_envelope_4}
	\Prob_{\bm{\xi}}\bigg(\max_{1\leq l\leq N(\delta_2)} \abs{\xi_l}\geq \frac{1}{2} \E_{\bm{\xi}} \max_{1\leq l\leq N(\delta_2)}  \abs{\xi_l} \bigg)\geq \bigg(\frac{ \E_{\bm{\xi}} \max_{1\leq l\leq N(\delta_2)}  \abs{\xi_l}   }{ 2 \left(\E_{\bm{\xi}} \max_{1\leq l\leq N(\delta_2)}  \abs{\xi_l}^r\right)^{1/r} }\bigg)^{r'}
	\end{align}
	for some conjugate indices $(r,r') \in (1,\infty)^2$. (\ref{ineq:lower_bound_envelope_2}) and (\ref{ineq:lower_bound_envelope_4}) suggest that we need to derive a lower bound for $\E_{\bm{\xi}} \max_{1\leq l\leq N(\delta_2)}  \abs{\xi_l} $ and an upper bound for $\E_{\bm{\xi}} \max_{1\leq l\leq N(\delta_2)}  \abs{\xi_l} ^r$. This can be done via the help of Lemma \ref{lem:characterization_maxima}: since $
	\Prob(\abs{\xi_1}>t)\asymp \frac{C_\epsilon}{1+t^{\alpha_\epsilon}}$  (cf. Property 1.2.15, page 16 of \cite{Samorodnitsky1994stable}),
	we can choose $\lambda\equiv 1$ and $\delta_0\asymp_{\epsilon} N(\delta_2)^{1/\alpha_\epsilon}$ to see that
	\begin{align*}
	\E_{\bm{\xi}} \max_{1\leq l\leq N(\delta_2)}  \abs{\xi_l}^r &\asymp \sum_{l=1}^{N(\delta_2)} \E \abs{\xi_l}^r\bm{1}_{\xi_l>\delta_0}\\
	&= N(\delta_2)\bigg( \Prob\left(\abs{\xi_1}>\delta_0\right) \int_0^{\delta_0} r u^{r-1}\ \d{u}\\
	&\qquad\qquad\qquad+\int_{\delta_0}^{\infty} ru^{r-1} \Prob\left(\abs{\xi_1}>u\right)\ \d{u}\bigg)\\
	&\asymp_{\epsilon, r} N(\delta_2)^{r/\alpha_\epsilon}.
	\end{align*}
	Now as long as $\epsilon<1/2$, we may choose $r>1$ close enough to $1$, e.g. $r=1.1$, to conclude that there exists $\mathfrak{p}_1 \in (0,1/8)$ that only depends on $\epsilon$ such that
	\begin{align}\label{ineq:lower_bound_envelope_5}
	\textrm{Left hand side of } (\ref{ineq:lower_bound_envelope_4})\geq 8\mathfrak{p}_1.
	\end{align}
	Combining (\ref{ineq:lower_bound_envelope_1}), (\ref{ineq:lower_bound_envelope_2}) and (\ref{ineq:lower_bound_envelope_5}), and the fact that $t_n= c_1 \big(\vartheta^{\gamma} n^{\frac{1-\gamma}{2-\gamma}}\big)^{2/\alpha_\epsilon}$ for some constant $c_1$ depending on $\epsilon,\gamma$ only, we have that for $n$ large enough depending on $\vartheta,\gamma$, 
	\begin{align}\label{ineq:lower_bound_envelope_6}
	\Prob\bigg( \sup_{f \in \tilde{\mathcal{F}}: Pf^2\leq \delta_2^2} \bigg\lvert\sum_{i=1}^n \xi_i f(X_i)\bigg\rvert\geq c_1 \big(\vartheta^{\gamma} n^{\frac{1-\gamma}{2-\gamma}} \big)^{2/\alpha_\epsilon}\bigg)&\geq 4\mathfrak{p}_1.
	\end{align}
	On the other hand, by Talagrand's concentration inequality (cf. Lemma \ref{lem:talagrand_conc_ineq}) and the contraction principle for Rademacher processes, we have with probability at least $1-2\mathfrak{p}_1$,
	\begin{align}\label{ineq:lower_bound_envelope_7}
	\sup_{f \in \tilde{\mathcal{F}}: Pf^2\leq \delta_2^2} \abs{\G_n(f^2)}&\leq C \big(\E \sup_{f \in \tilde{\mathcal{F}}: Pf^2\leq \delta_2^2} \abs{\G_n f} +\delta_2 \sqrt{ \log(1/2\mathfrak{p}_1) }+\log(1/2\mathfrak{p}_1)/\sqrt{n}\big)\\
	&\leq C_\epsilon \cdot \delta_2\sqrt{\log (1/\delta_2)} \leq C_{\epsilon,\gamma} \vartheta n^{-\frac{1}{2(2-\gamma)}} \sqrt{\log n}.\nonumber
	\end{align}
	Combining (\ref{ineq:lower_bound_envelope_6})-(\ref{ineq:lower_bound_envelope_7}), we see that with probability at least $2\mathfrak{p}_1$, 
	\begin{align*}
	&\sup_{f \in \tilde{\mathcal{F}}: Pf^2\leq \delta_2^2} (\Prob_n -P)(2\xi f-f^2)\\
	&\geq 2 c_1 \big(\vartheta^{\gamma} n^{\frac{1-\gamma}{2-\gamma}} \big)^{2/\alpha_\epsilon} \cdot n^{-1} -C_{\epsilon,\gamma} \vartheta n^{-\frac{1}{2(2-\gamma)}-\frac{1}{2}} \sqrt{\log n}\\
	& \geq  2c_1 \vartheta^{\gamma} n^{-\frac{1 }{2-\gamma}}\cdot \tau_n(\epsilon,\gamma)-C_{\epsilon,\gamma} \vartheta n^{-\frac{(3-\gamma)/2}{(2-\gamma)}} \sqrt{\log n}\geq c_1\vartheta^{\gamma} n^{-\frac{1 }{2-\gamma}}\cdot \tau_n(\epsilon,\gamma)
	\end{align*}
	for $n$ large enough depending on $\epsilon,\vartheta,\gamma$, where $\tau_n(\epsilon,\gamma)\equiv n^{\frac{1-\gamma}{2-\gamma}\cdot \frac{\epsilon}{2-\epsilon}}$. Hence with the same probability estimate, 
	\begin{align*}
	F_n(\delta_2)&= \sup_{f \in \tilde{\mathcal{F}}: Pf^2\leq \delta_2^2} (\Prob_n -P)(2\xi f-f^2)-\delta_2^2\\
	&\geq c_1 \vartheta^{\gamma} n^{-\frac{1 }{2-\gamma}} \tau_n(\epsilon,\gamma)- \vartheta^2 n^{-\frac{1 }{2-\gamma}} \geq \frac{1}{2} c_1\vartheta^{\gamma} n^{-\frac{1}{2-\gamma}}\tau_n(\epsilon,\gamma) 
	\end{align*}
	holds for $n$ large enough depending on $\epsilon,\vartheta,\gamma$, completing the proof for the claim for $0< \gamma<1$.
	
	\noindent \textbf{(Case 2: $\gamma =1$).} Recall $\delta_2 =\vartheta n^{-1/2}$, and there exists one interval $I$ with length $\delta_2^2$. It is easy to see that $
	\Prob\big(\abs{\sum_{i=1}^n \bm{1}_I(X_i)-\vartheta^2}>\vartheta^2/2\big)\leq 2\exp(-\vartheta^2/10)$. 
	For $\vartheta \geq 4$, we see that with probability at least $0.5$,
	there are $\mathcal{O}(1)$ points $X_i\in I$. Denote this event $\mathcal{E}_1$. Let
	\begin{align*}
	Z_n\equiv \sup_{f \in \tilde{\mathcal{F}}: Pf^2\leq \delta_2^2 }\bigg(2 \sum_{i=1}^n \xi_i f(X_i)-n\left(\Prob_n-P\right)(f^2)\bigg).
	\end{align*}
	Note we can use the absolute value in the suprema in the above display. Since $
	\E \abs{\left(\Prob_n-P\right)(\bm{1}_I)}^2\leq \vartheta^2 n^{-2}$, we see that on an event with probability at least $0.96$, $\abs{n\left(\Prob_n-P\right)(\bm{1}_I)}\leq 25 \vartheta$. Denote this event by $\mathcal{E}_2$. Then for any $\xi$ such that $\E \abs{\xi}\geq 25\vartheta$, let $t=\E \abs{\xi}-25\vartheta$, and $N_I\equiv \sum_{i=1}^n \bm{1}_I(X_i)$,
	\begin{align*}
	\Prob\left(Z_n\geq t\right)&\geq \E_{\bm{X}} \bigg[ \Prob_{\bm{\xi}} \bigg(\bigg\lvert 2 \sum_{i=1}^n \xi_i \bm{1}_I (X_i)-n\left(\Prob_n-P\right)(\bm{1}_I)\bigg\rvert\geq t \bigg) \bm{1}_{\mathcal{E}_1\cap \mathcal{E}_2}\bigg]\\
	&\geq \E_{\bm{X}}\bigg[\Prob_{\bm{\xi}}\bigg( \biggabs{\sum_{i=1}^{N_I} \xi_i  }> (t+25\vartheta)/2 \bigg)\bm{1}_{\mathcal{E}_1\cap \mathcal{E}_2}\bigg]\\
	&\geq \E_{\bm{X}}\bigg[\Prob_{\bm{\xi}}\bigg( \biggabs{\sum_{i=1}^{N_I} \xi_i  }> \frac{1}{2}\E_{\bm{\xi}} \biggabs{\sum_{i=1}^{N_I} \xi_i  }\bigg)\bm{1}_{\mathcal{E}_1\cap \mathcal{E}_2}\bigg]
	\end{align*}
	where in the last inequality we used Jensen's inequality. Let $\eta$ be a symmetric random variable given by $\Prob(\abs{\eta}>t)=1/(1+t^2)$, then it is easy to calculate that $\E \abs{\eta}=\pi/2$, and $\E\abs{\eta}^r\equiv c_r<\infty$ for $r<2$. Let $\xi\equiv 50\vartheta^2 \cdot \eta$. Then $\E \abs{\xi}=25\pi \vartheta^2>25\vartheta$, and hence choosing $r>1$ close enough to $1$ in the Paley-Zygmund Lemma \ref{lem:paley_zygmund} yields that
	\begin{align*}
	\Prob\left(Z_n\geq 25\pi\vartheta^2-25\vartheta\right) \geq 2\mathfrak{p}_2
	\end{align*}
	for some constant $\mathfrak{p}_2>0$ depending only on $\vartheta$ (through the estimate on $N_I$ on the event $\mathcal{E}_1$). Hence with probability at least $2\mathfrak{p}_2$, 
	\begin{align*}
	F_n(\delta_2) \geq (25\pi\vartheta^2-25\vartheta) n^{-1} - \vartheta^2 n^{-1}\geq 285 \vartheta n^{-1}.
	\end{align*}
	This completes the proof.
\end{proof}

\subsubsection{Proof of claim (2) of Proposition \ref{prop:lower_bound_envelope}} 

\begin{proof}[Proof of Proposition \ref{prop:lower_bound_envelope}, claim (2)]
	Recall $\delta_1 \equiv \rho n^{-\frac{1}{2(2-\gamma)}-\beta_\epsilon}$. Note that by Koltchinskii-Pollard maximal inequality for empirical processes (cf. Theorem 2.14.1 of \cite{van1996weak}), we have 
	\begin{align*}
	\max_{1\leq k\leq n}\E \sup_{f \in \tilde{\mathcal{F}}: Pf^2\leq \delta_1^2} \bigg\lvert\frac{1}{\sqrt{k}}\sum_{i=1}^k \epsilon_i f(X_i) \bigg\rvert \lesssim \pnorm{\tilde{F}(\delta_1)}{L_2(P)}\leq C_1\delta_1^\gamma.
	\end{align*}
	Hence we may take $
	\psi_n(k)\equiv C_1k^{1/(2-2\epsilon)} \delta_1^\gamma$ 
	in the multiplier inequality Lemma \ref{lem:multiplier_ineq} to see that
	\begin{align*}
	\E \sup_{f \in \tilde{\mathcal{F}}: Pf^2\leq \delta_1^2} \bigg\lvert\sum_{i=1}^n \xi_i f(X_i) \bigg\lvert&\leq 4 \int_0^\infty \psi_n\big(n \Prob(\abs{\xi_1}>t)\big)\ \d{t}\\
	&\leq 4C_1 \delta_1^\gamma n^{1/2(1-\epsilon)} \pnorm{\xi_1}{2(1-\epsilon),1}.
	\end{align*}
	On the other hand, again by the Koltchinskii-Pollard maximal inequality and the contraction principle for Rademacher processes, 
	\begin{align*}
	\E \sup_{f \in \tilde{\mathcal{F}}: Pf^2\leq \delta_1^2} \abs{\G_n (f^2)} &\lesssim \E \sup_{f \in \tilde{\mathcal{F}}: Pf^2\leq \delta_1^2} \biggabs{ \frac{1}{\sqrt{n}}\sum_{i=1}^n \epsilon_i f(X_i)}\lesssim \delta_1^\gamma.
	\end{align*}
	Combining the above estimates, we arrive at
	\begin{align*}
	\E E_n(\delta_1) &\leq 8C_1 \delta_1^\gamma n^{-1}\cdot n^{1/2(1-\epsilon)} \pnorm{\xi_1}{2(1-\epsilon),1}+ C_2  n^{-1/2} \delta_1^\gamma\\
	&\leq C_{\epsilon,\xi} \rho^\gamma n^{-\frac{1}{2-\gamma}} \omega_n(\epsilon,\gamma).
	\end{align*}
	The claim (2) of Proposition \ref{prop:lower_bound_envelope} now follows from Markov's inequality and hence the proof of Theorem \ref{thm:envelope_rate_lower_bound} is complete.
\end{proof}

\subsection{Proof of entropy results}

\subsubsection{Proof of Lemma \ref{lem:subset_VC}}

\begin{proof}[Proof of Lemma \ref{lem:subset_VC}]
	Let $t_j\equiv (1+\epsilon)^{-j}$ and $m(\epsilon)$ be the smallest integer $j$ such that $t_j\leq \epsilon \pnorm{F}{L_2(Q)}$. Now for any $f \in \mathcal{F}$, define
	\begin{align*}
	f_\epsilon& \equiv \sum_{j=1}^{m(\epsilon)} \big( t_j \bm{1}_{t_j<f\leq t_{j-1}} + (-t_{j-1}) \bm{1}_{-t_{j-1}< f\leq -t_{j}} \big).
	\end{align*}
	Then if $x \in \mathcal{X}$ is such that 
	\begin{enumerate}
		\item $t_j<f(x)\leq t_{j-1}$ for some $j \leq m(\epsilon)$, 
		\begin{align*}
		0\leq f(x)-f_\epsilon(x)\leq t_{j-1}-t_j\leq \epsilon t_j\leq \epsilon f(x)\leq \epsilon F(x).
		\end{align*}
		\item $-t_{j-1}<f(x)\leq -t_j$ for some $j \leq m(\epsilon)$,
		\begin{align*}
		0\leq f(x)-f_\epsilon(x)\leq -t_j-(-t_{j-1})\leq \epsilon t_j\leq \epsilon \big(-f(x) \big)\leq \epsilon F(x).
		\end{align*}
		\item $-t_{m(\epsilon)}<f(x)\leq t_{m(\epsilon)}$,
		\begin{align*}
		\abs{f(x)-f_\epsilon(x)}\leq t_{m(\epsilon)}\leq \epsilon \pnorm{F}{L_2(Q)}.
		\end{align*}
	\end{enumerate}
	Combining the above discussion we arrive at $
	\pnorm{f-f_\epsilon}{L_2(Q)}^2\leq 3\epsilon^2 \pnorm{F}{L_2(Q)}^2$. Let $\mathcal{F}_\epsilon\equiv \{f_\epsilon: f \in \mathcal{F}\}$. Then since the sets
	\begin{align*}
	\big\{(x,t): f_\epsilon(x)\geq t\big\}&=\cup_{j=1}^{m(\epsilon)} \big\{x: f(x)\geq t_{j}\big\}\times (t_j, t_{j-1}] \\
	&\qquad \bigcup \cup_{j=1}^{m(\epsilon)} \big\{x: f(x)\geq -t_{j-1}\big\}\times (-t_{j-1}, -t_{j}] \\
	&\qquad\qquad \bigcup  \big\{x: f(x)\geq -t_{m(\epsilon)}\big\}\times (-t_{m(\epsilon)}, t_{m(\epsilon)}]
	\end{align*}
	as $f_\epsilon$ ranges over $\mathcal{F}_\epsilon$ is the union of at most $2m(\epsilon)+1$ VC-classes with disjoint supports, and hence the VC-dimension of $\mathcal{F}_\epsilon$ is no larger than $Vm(\epsilon)$, where $V\in (0,\infty)$ only depends on $\mathcal{F}_0$. The rest of the proof proceeds along the same lines as in page 1172 of \cite{gine2006concentration}.
\end{proof}

\subsubsection{Proof of Lemma \ref{lem:entropy_ind}}

\begin{proof}[Proof of Lemma \ref{lem:entropy_ind}]
	Let $\{h_i\}_{i=1}^N$ be a minimal $\epsilon$-covering set of $\mathcal{H}$ under $L_\infty$. For any probability measure $Q$ on $[0,1]^2$, and any $f\otimes (h-g) \in f\otimes \big(\mathcal{H}-g\big)$, take $h_i$ such that $\pnorm{h-h_i}{\infty}\leq \epsilon$. Then
	\begin{align*}
    \pnorm{f\otimes (h_i-g) - f\otimes (h-g)}{L_2(Q)}^2&\leq \pnorm{f}{L_2(Q)}^2 \epsilon^2 = \pnorm{f\otimes 1}{L_2(Q)}^2\epsilon^2.
	\end{align*}
	completing the proof.
\end{proof}

\subsubsection{Proof of Lemma \ref{lem:cond_H2_1}}

\begin{proof}[Proof of Lemma \ref{lem:cond_H2_1}]
	Since $\mathcal{F}'\subset \mathcal{F}\cap L_\infty(1)$ is VC-major, Lemma \ref{lem:subset_VC} yields that for any probability measure $Q_x$ on $[0,1]$ and any $\epsilon>0$,
		\begin{align*}
		\log \mathcal{N}\big(\epsilon \pnorm{F'}{L_2(Q_x)}, \mathcal{F}', L_2(Q_x) \big)\leq \frac{C}{\epsilon}  \log \bigg(\frac{C}{\epsilon}\bigg)  \log\bigg(\frac{1}{\epsilon\pnorm{F'}{L_2(Q_x)} }\bigg).
		\end{align*}
	Now for any discrete probability measure $Q=n^{-1}\sum_{i=1}^n \delta_{(x_i,z_i)}$ on $[0,1]^2$, let $Q_x\equiv n^{-1}\sum_{i=1}^n \delta_{x_i}$ be the (marginal) probability measure on $[0,1]$. Take a minimal $\epsilon\pnorm{F'}{L_2(Q_x)}$-cover of $\mathcal{F}'$ under $L_2(Q_x)$, namely $\{f_k\}$, the log-cardinality of which is no more than 
		\begin{align*}
		\frac{C}{\epsilon}  \log \bigg(\frac{C}{\epsilon}\bigg)  \log\bigg(\frac{1}{\epsilon\pnorm{F'}{L_2(Q_x)} }\bigg).
		\end{align*}
	Further take a minimal $\epsilon$-cover of $\mathcal{H}$ under $L_\infty$, namely $\{h_l\}$, the log-cardinality of which is at most a constant multiple of $\epsilon^{-\gamma}$. Consider the set $\{f_k\otimes h_l\}$, the log-cardinality of which is at most a constant multiple of 
		\begin{align*}
		\frac{1}{\epsilon} \log \bigg(\frac{1}{\epsilon}\bigg)  \log\bigg(\frac{1}{\epsilon\pnorm{F'\otimes 1}{L_2(Q)} }\bigg) \vee \epsilon^{-\gamma}.
		\end{align*}
	For every $f\otimes (h-(\phi_0-f_0)) \in \mathcal{F}'\otimes (\mathcal{H}-(\phi_0-f_0))$, let $\tilde{f}_k,\tilde{h}_l$ be such that $\pnorm{f-\tilde{f}_k}{L_2(Q_x)}\leq \epsilon \pnorm{F'}{L_2(Q_x)}$ and $\pnorm{h-\tilde{h}_l}{\infty}\leq \epsilon$. Then
	\begin{align*}
	&\pnorm{f\otimes (h-(\phi_0-f_0))-\tilde{f}_k\otimes (\tilde{h}_l-(\phi_0-f_0))}{L_2(Q)}^2\\
	&=\frac{1}{n}\sum_{i=1}^n \bigg(f(x_i)\big(h(z_i)-(\phi_0(x_i,z_i)-f_0(x_i))\big)\\
	&\qquad\qquad\qquad-\tilde{f}_k(x_i)\big(\tilde{h}_l(z_i)-(\phi_0(x_i,z_i)-f_0(x_i))\big)\bigg)^2 \\
	&\lesssim \pnorm{h-\tilde{h}_l}{\infty}^2 \pnorm{F'}{L_2(Q_x)}^2+ \pnorm{\phi_0}{\infty}^2\pnorm{f-\tilde{f}_k}{L_2(Q_x)}^2\\
	&\lesssim (1\vee \pnorm{\phi_0}{\infty}^2)\epsilon^2 \pnorm{F'}{L_2(Q_x)}^2 = (1\vee \pnorm{\phi_0}{\infty}^2)\epsilon^2  \pnorm{F'\otimes 1}{L_2(Q)}^2,
	\end{align*}
	as desired.
\end{proof}

\subsection{Proof of stochastic boundedness of shape-restricted LSEs}

\subsubsection{Proof of Lemma \ref{lem:uniform_bound_iso}}

\begin{proof}[Proof of Lemma \ref{lem:uniform_bound_iso}, isotonic case]
	The isotonic least squares estimator $\hat{f}_n$ has a well-known min-max representation \cite{robertson1988order}:
	\begin{align}\label{ineq:represent_iso}
	\hat{f}_n({X}_j)=\min_{v\geq j}\max_{u\leq j}\frac{1}{v-u+1}\sum_{i=u}^v {Y}_i
	\end{align}
	where we slightly abuse the notation $X_i$'s so that ${X}_1\leq \cdots \leq {X}_n$ denote the ordered covariates and ${Y}_i$ denotes the corresponding observed response at ${X}_i$.
	Since $\hat{f}_n$ is non-decreasing, we only need to consider 
	\begin{align*}
	\alpha_n \equiv \hat{f}_n({X}_1) = \min_{v\geq 1} \frac{1}{v}\sum_{i=1}^v{Y}_i,\quad \beta_n \equiv \hat{f}_n({X}_n) = \max_{u\leq n} \frac{1}{n-u+1}\sum_{i=u}^n {Y}_i.
	\end{align*}
	Note that 
	\begin{align*}
	\E \abs{\alpha_n} \vee \E \abs{\beta_n}\leq  \E\max_{k\leq n}\biggabs{\frac{1}{k}\sum_{i=1}^k \xi_i} + \pnorm{f_0}{\infty}.
	\end{align*}
	The first term is $\mathcal{O}(1)$ by a simple blocking argument and a L\'evy-type maximal inequality due to Montgomery-Smith \cite{montgomery1993comparison} (see also Theorem 1.1.5 of \cite{de2012decoupling}); we include some details for the convenience of the reader: suppose without loss of generality that $\log_2 n$ is an integer, then for any $t\geq 1$, 
	\begin{align*}
	\Prob\bigg(\max_{1\leq k\leq n} \biggabs{\frac{1}{k}\sum_{i=1}^k \xi_i}>t\bigg)&\leq \sum_{j=1}^{\log_2 n} \Prob\bigg(\max_{2^{j-1}\leq k<2^j} \biggabs{\frac{1}{k}\sum_{i=1}^k \xi_i}> t\bigg)+\Prob\bigg(\biggabs{\sum_{i=1}^n \xi_i}>n t\bigg)\\
	&\leq \sum_{j=1}^{\log_2 n} \Prob\bigg(\max_{2^{j-1}\leq k<2^j} \biggabs{\sum_{i=1}^k \xi_i}> 2^{j-1} t\bigg)+\frac{\pnorm{\xi_1}{2}^2}{nt^2}\\
	&\leq 9\sum_{j=1}^{\log_2 n} \Prob\bigg( \biggabs{\sum_{i=1}^{2^j} \xi_i}> 2^{j-1} t/30\bigg)+\frac{\pnorm{\xi_1}{2}^2}{nt^2}\\
	&\leq C \pnorm{\xi_1}{2}^2\bigg(\sum_{j=1}^{\log_2 n}\frac{1}{2^j t^2}+\frac{1}{nt^2}\bigg)\leq C'\pnorm{\xi_1}{2}^2 t^{-2},
	\end{align*}
	completing the proof.
\end{proof}

The proof of stochastic boundedness of the convex least squares estimator crucially uses the characterization developed in Lemma 2.6 of \cite{groeneboom2001estimation}. Note that the characterization is purely deterministic.

\begin{lemma}\label{lem:char_cvx_reg}
$\hat{f}_n$ is a convex least squares estimator if and only if for all $j=2,\ldots,n$, 
	\begin{align*}
	\sum_{k=1}^{j-1} R_k(X_{k+1}-X_k) \geq \sum_{k=1}^{j-1} S_k(X_{k+1}-X_k),
	\end{align*}
	with inequality holds if and only if $\hat{f}_n$ has a kink at $X_j$. Here $R_k = \sum_{i=1}^k \hat{f}_n(X_i)$ and $S_k = \sum_{i=1}^k Y_i$, where we abuse the notation $X_i$'s for the ordered covariates such that $X_1\leq \ldots\leq X_n$, and $Y_i$'s are the corresponding observed responses at $X_i$.
\end{lemma}

\begin{proof}[Proof of Lemma \ref{lem:uniform_bound_iso}, convex case]
	By symmetry we only consider the behavior of $\hat{f}_n(0)$. Let $\tau_n$ denote the first kink of $\hat{f}_n$ away from $0$. Then it follows from the characterization Lemma \ref{lem:char_cvx_reg} that
	\begin{align*}
	\sum_{k=1}^{\tau_n-2}R_k(X_{k+1}-X_k) & \geq \sum_{k=1}^{\tau_n-2} S_k(X_{k+1}-X_k),\\
	\sum_{k=1}^{\tau_n-1}R_k(X_{k+1}-X_k) & = \sum_{k=1}^{\tau_n-1} S_k(X_{k+1}-X_k).
	\end{align*}
	The above two (in)equalities necessarily entail that 
	\begin{align*}
	R_{\tau_n-1}(X_{\tau_n}-X_{\tau_{n}-1})\leq S_{\tau_n-1}(X_{\tau_n}-X_{\tau_{n}-1}).
	\end{align*}
	 Hence with probability $1$ we have $R_{\tau_n-1}\leq S_{\tau_n-1}$, i.e.
	\begin{align}\label{ineq:upper_bound_cvx_origin_1}
	 \sum_{i=1}^{\tau_n-1} \hat{f}_n(X_i)\leq \sum_{i=1}^{\tau_n-1} Y_i.
	\end{align}
	Since $\hat{f}_n$ is linear on $[0,X_{\tau_n}]$, we can write 
	\begin{align}\label{ineq:upper_bound_cvx_origin_2}
	\hat{f}_n(x) = \bigg(1-\frac{x}{X_{\tau_n}}\bigg) \hat{f}_n(0) + \frac{x}{X_{\tau_n}} \hat{f}_n(X_{\tau_n}).
	\end{align}
	Combining (\ref{ineq:upper_bound_cvx_origin_1}) and (\ref{ineq:upper_bound_cvx_origin_2}) we see that
	\begin{align*}
	\bigg[\sum_{i=1}^{\tau_n-1}\bigg(1-\frac{X_i}{X_{\tau_n}}\bigg)\bigg] \hat{f}_n(0) + \bigg[\sum_{i=1}^{\tau_n-1}\frac{X_i}{X_{\tau_n}}\bigg] \hat{f}_n(X_{\tau_n}) \leq \sum_{i=1}^{\tau_n-1} Y_i,
	\end{align*}
	and hence
	\begin{align}\label{ineq:upper_bound_cvx_origin_3}
	\hat{f}_n(0) \leq \bigg(\frac{1}{1-\beta_{\tau_n}}\bigg)\cdot \frac{\sum_{i=1}^{\tau_n-1}Y_i}{\tau_n-1} + \frac{\beta_{\tau_n}}{1-\beta_{\tau_n}}\bigabs{\inf_{x \in [0,1]} \hat{f}_n(x)},
	\end{align}
	where 
	\begin{align*}
	\beta_k = \bigg(\frac{1}{k-1}\sum_{i=1}^{k-1} X_i\bigg)\cdot \frac{1}{X_{k}}.
	\end{align*}
	By (\ref{ineq:upper_bound_cvx_origin_3}), we need to handle three terms: 
	\begin{enumerate}
	\item[(i)] $(1-\beta_{\tau_n})^{-1}$, 
	\item[(ii)] $\frac{\sum_{i=1}^{\tau_n-1}Y_i}{\tau_n-1}$, and
	\item[(iii)] $\abs{\inf_{x \in [0,1]} \hat{f}_n(x)}$. 
	\end{enumerate}
	We first handle term (i). We claim that for some universal constant $C>0$, it holds that
	\begin{align}\label{ineq:upper_bound_cvx_origin_4}
	\Prob\big(\max_{2\leq k\leq n}(1-\beta_k)^{-1}\geq t\big)\leq C t^{-1}.
	\end{align}
	To see this, note that for each $k \leq n$, conditional on $X_k$, $X_1/X_k,\ldots,X_{k-1}/X_k$ are distributed as the order statistics for $k-1$ uniform random variables on $[0,1]$. Let $U_1,\ldots,U_n$ be an i.i.d. sequence of uniformly distributed random variables on $[0,1]$, and $0\leq U_{(1)}^n\leq \ldots\leq U_{(n)}^n\leq 1$ be their associated order statistics. Then by using a union bound, the probability in (\ref{ineq:upper_bound_cvx_origin_4}) is bounded by
	\begin{align*}
	\sum_{k=2}^n \Prob \bigg(\frac{1}{k-1}\sum_{i=1}^{k-1}\frac{X_i}{X_k}\geq 1-t^{-1}\bigg)\leq \sum_{k=1}^{n-1} \E \bigg[\Prob\bigg(\frac{1}{k}\sum_{j=1}^{k}U_{(j)}^k\geq 1-t^{-1}\bigg)\bigg\lvert X_{k+1}\bigg].
	\end{align*}
	For $t\geq 3$, the probability in the bracket  equals $
	\Prob\big(\sum_{j=1}^k U_j\leq kt^{-1}\big)=\frac{(kt^{-1})^k}{k!}$
	by volume computation: $\abs{\{\sum_{j=1}^k x_j\leq a\}}=a^k/k!$. Now combining the probability estimates we arrive at
	\begin{align*}
	\Prob\big(\max_{2\leq k\leq n}(1-\beta_k)^{-1}\geq t\big)\leq \sum_{k\geq 1} \frac{(kt^{-1})^k}{k!} \leq  \sum_{k\geq 1} \frac{(kt^{-1})^k}{(k/e)^k}\leq Ct^{-1},
	\end{align*}
	proving the claim (\ref{ineq:upper_bound_cvx_origin_4}) for $t\geq 3$. For $t<3$, it suffices to increase $C$.

	The second term (ii) can be handled along the same lines as in the proof for the isotonic model, assuming $\pnorm{f_0}{\infty}<\infty$ and $\pnorm{\xi_1}{2}<\infty$.
	
	Finally we consider the third term (iii) $\abs{\inf_{x \in [0,1]} \hat{f}_n(x)}$. We claim that with probability $1$,
    \begin{align}\label{ineq:upper_bound_cvx_origin_5}
    \limsup_{n \to \infty}\sup_{x \in [1/4,3/4]}\abs{\hat{f}_n(x)}\leq C_{\xi,f_0}.
    \end{align}
    The claim will be verified in the proof of Lemma \ref{lem:uniform_bound_additive_iso} below in a more general setting. In particular, (\ref{ineq:upper_bound_cvx_origin_5}) implies that $\sup_{x \in [1/4,3/4]}\abs{\hat{f}_n(x)}=\mathcal{O}_{\mathbf{P}}(1)$. Hence for any $\epsilon>0$, there exists a constant $K_\epsilon>0$ such that for all $n$ large enough, with probability at least $1-\epsilon$, $\sup_{x \in [1/4,3/4]}\abs{\hat{f}_n(x)}\leq K_\epsilon$. This event is denoted $\mathcal{E}_\epsilon$. Now by convexity of $\hat{f}_n$, it follows that $\abs{\inf_{x \in [0,1]} \hat{f}_n(x)}\leq 2K_\epsilon$ on $\mathcal{E}_\epsilon$. To see this, we only need to consider the case where the minimum of $\hat{f}_n$ is attained in, say, $[0,1/4]$: then the line connecting $(1/4,\hat{f}_n(1/4))$ and $(3/4,\hat{f}_n(3/4))$ minorizes $\hat{f}_n$ on $[0,1/4]$, which is bounded from below by $-2K_\epsilon$ and hence the same lower bound holds for $\inf_{x \in [0,1]}\hat{f}_n(x)$ on the event $\mathcal{E}_\epsilon$. An upper bound for $\inf_{x \in [0,1]}\hat{f}_n(x)$ is trivial: $\inf_{x \in [0,1]}\hat{f}_n(x)\leq \sup_{x \in [1/4,3/4]}\hat{f}_n(x)\leq K_\epsilon$ on $\mathcal{E}_\epsilon$. These arguments complete the proof for $\abs{\inf_{x \in [0,1]} \hat{f}_n(x)}=\mathcal{O}_{\mathbf{P}}(1)$.

	The claim that $\pnorm{\hat{f}_n}{\infty}=\mathcal{O}_{\mathbf{P}}(1)$ follows by combining the discussion of the three terms above and (\ref{ineq:upper_bound_cvx_origin_3}) which proved $\abs{\hat{f}_n(0)}\vee \abs{\hat{f}_n(1)}=\mathcal{O}_{\mathbf{P}}(1)$ and $\abs{\inf_{x \in [0,1]} \hat{f}_n(x)}=\mathcal{O}_{\mathbf{P}}(1)$.
\end{proof}

\subsubsection{Proof of Lemma \ref{lem:uniform_bound_additive_iso}}

\begin{proof}[Proof of Lemma \ref{lem:uniform_bound_additive_iso}, isotonic case]
	The proof essentially follows the isotonic case of Lemma \ref{lem:uniform_bound_iso} by noting that the least squares estimator $\hat{f}_n$ for $\mathcal{F}$ in the additive model  has the following representation:
	\begin{align*}
	\hat{f}_n(X_j)=\min_{v\geq j}\max_{u\leq j}\frac{1}{v-u+1}\sum_{i=u}^v \big({Y}_i-\hat{h}_n({Z}_i)\big)
	\end{align*}
	where ${X}_1\leq \cdots \leq  {X}_n$ denote the ordered $X_i$'s, ${Y}_i$'s are the observed responses at the corresponding ${X}_i$'s, and ${Z}_i$'s are the corresponding $Z_i$'s following the ordering of the ${X}_i$'s. The rest of the proof proceeds along the same lines as in the isotonic case of Lemma \ref{lem:uniform_bound_iso} by noting that
	\begin{align}\label{ineq:uniform_bound_additive_isotonic_1}
	&\max_{1\leq k\leq n} \sup_{h \in \mathcal{H}} \biggabs{\frac{1}{k}\sum_{i=1}^k \big(\phi_0(X_i, Z_i)-h(Z_i)}\\
	&\leq 	\pnorm{\phi_0}{\infty}+ \max_{1\leq k\leq n} \bigg(\frac{1}{k}\sum_{i=1}^k H(Z_i)\bigg)=\mathcal{O}_{\mathbf{P}}(1),\nonumber
	\end{align}
	where the stochastic boundedness follows from the same arguments using L\'evy-type maximal inequality as in the isotonic case of Lemma \ref{lem:uniform_bound_iso}, since we have assumed $P_Z H^2<\infty$.
\end{proof}

\begin{proof}[Proof of Lemma \ref{lem:uniform_bound_additive_iso}, convex case]
We use the same strategy as the convex case of Lemma \ref{lem:uniform_bound_iso}  by replacing $Y_i$ with $Y_i-\hat{h}_n(Z_i)$, and handling terms (i), (ii) and (iii) as in the proof of the convex case of Lemma \ref{lem:uniform_bound_iso} . Term (i) can be handled using the same arguments as in the proof of the convex case of Lemma \ref{lem:uniform_bound_iso} ; term (ii) can be handled similar to (\ref{ineq:uniform_bound_additive_isotonic_1}). Hence it remains to handle (iii). Let $\hat{\phi}_n(x,z)\equiv \hat{f}_n(x)+\hat{h}_n(z)$. We claim that there exists some $M>0$ such that
\begin{align}\label{ineq:uniform_bound_additive_convex_1}
\Prob\bigg( \inf_{(x,z) \in [1/4,3/4]^2} \abs{\hat{\phi}_n(x,z)-\phi_0(x,z)}>M \textrm{ i.o.}\bigg)=0.
\end{align}
Once (\ref{ineq:uniform_bound_additive_convex_1}) is proved, the event $\mathcal{E}\equiv \cup_{m\geq 1}\cap_{n\geq m} \{ \inf_{x \in [1/4,3/4]} \abs{\hat{f}_n(x)}\leq \bar{M}\}$ happens with probability $1$, where $\bar{M}\equiv M+\sup_{(x,z)\in[1/4,3/4]^2} H(z)+\pnorm{\phi_0}{\infty}<\infty$. Let $x_n \in \argmin_{x \in [1/4,3/4]} \hat{f}_n(x)$ and $M_n\equiv \abs{\hat{f}_n(x_n)}$. On the event $\mathcal{E}$, for all $n$ large enough, there exists $x^\ast_{n} \in [1/4,3/4]$ such that $\abs{\hat{f}_{n}(x^\ast_{n})}\leq 2\bar{M}$. The key observation is the following: if $M_n>10\bar{M}$, then 
\begin{align}\label{ineq:uniform_bound_additive_convex_8}
\inf_{x \in [1/16,1/8]} \hat{f}_{n}(x) \vee \inf_{x \in [7/8,15/16]}  \hat{f}_{n}(x)
&\geq \frac{1}{4}\big(M_n-10\bar{M}\big).
\end{align}
To see this, we only consider the case $1/4\leq x_n< x_n^\ast\leq 3/4$, and derive a lower bound for $\inf_{x \in [7/8,15/16]} \hat{f}_{n}(x)$; the other case follows from similar arguments. Note that the line $L$ connecting $(x_n, \hat{f}_n(x_n))$ and $(x_n^\ast, \hat{f}_n(x_n^\ast))$  minorizes $\hat{f}_n$ on $[7/8,15/16]$. Since $M_n>10\bar{M}>2\bar{M}$, $\hat{f}_n(x_n)<0$ and hence the line $L$ has a positive slope $s_L$ bounded below by $(M_n-2\bar{M})/(3/4-1/4)=2(M_n-2\bar{M})$. This implies that for any $x \in [7/8, 15/16]$,
\begin{align*}
\hat{f}_n(x)\geq \hat{f}_n(7/8)\geq L(7/8)&=L(x_n^\ast)+s_L(7/8-x_n^\ast)\\
&\geq  \hat{f}_n(x_n^\ast) + 2(M_n-2\bar{M})\cdot (7/8-3/4)\\
&\geq (-2\bar{M})+\frac{1}{4}(M_n-2\bar{M})=\frac{1}{4}\big(M_n-10\bar{M}\big),
\end{align*}
proving (\ref{ineq:uniform_bound_additive_convex_8}). Now we assume without loss of generality that $\inf_{x \in [1/16,1/8]} \hat{f}_{n}(x) \geq (M_n-10\bar{M})/4$. Let $I\equiv [1/16,1/8]\times [0,1]$. Since
\begin{align*}
&\frac{1}{n}\sum_{i=1}^{n} \big(Y_i-\hat{\phi}_{n}(X_i,Z_i)\big)^2\\
& = \frac{1}{n} \sum_{i=1}^{n} \big(\xi_i+\phi_0(X_i,Z_i)-\hat{h}_{n}(Z_i)-\hat{f}_{n}(X_i)\big)^2\\
&\geq \frac{1}{n}\sum_{(X_i,Z_i) \in I}\big(\hat{f}_{n}(X_i)-\big(H(Z_i)+\pnorm{\phi_0}{\infty}+\abs{\xi_i}\big)\big)_+^2\\
&\geq \frac{1}{2n}\sum_{(X_i,Z_i) \in I } \hat{f}_{n}^2(X_i) - \frac{1}{n} \sum_{(X_i,Z_i) \in I} \big(3H^2(Z_i)+3\pnorm{\phi_0}{\infty}^2+3\xi_i^2\big)\\
& \geq \bigg(\frac{(M_n-10\bar{M})^2 }{32} -3\pnorm{\phi_0}{\infty}^2 \bigg)\frac{\abs{\{ i \in [1:n]: (X_i,Z_i) \in I\}}}{n}\\
&\qquad\qquad\qquad\qquad - \frac{3}{n}\sum_{(X_i,Z_i) \in I} H^2(Z_i)- \frac{3}{n}\sum_{(X_i,Z_i) \in I} \xi_i^2.
\end{align*}
Hence by the law of large numbers, on an event with probability $1$, if $M_n>10\bar{M}$,
\begin{align}\label{ineq:uniform_bound_additive_convex_2}
& \limsup_{n \to \infty} \frac{1}{n}\sum_{i=1}^{n} \big(Y_i-\hat{\phi}_{n}(X_i,Z_i)\big)^2\\
&\qquad \qquad \geq \frac{ (\limsup_{n \to \infty} M_n-10\bar{M})^2 }{16\cdot 32} -\frac{3}{16}\big(\pnorm{\phi_0}{\infty}^2+P_Z H^2 + \E \xi_1^2\big).\nonumber
\end{align}
On the other hand, since $\hat{\phi}_{n}$ is the least squares estimator, for any $h' \in \mathcal{H}$,
\begin{align}\label{ineq:uniform_bound_additive_convex_3}
&\limsup_{n \to \infty} \frac{1}{n}\sum_{i=1}^{n} \big(Y_i-\hat{\phi}_{n}(X_i,Z_i)\big)^2 \\
&\leq \limsup_{n \to \infty} \frac{1}{n}\sum_{i=1}^{n} \big(Y_i-h'(Z_i)\big)^2\leq 3\E \xi_1^2+3\pnorm{\phi_0}{\infty}^2+3P_Z H^2.\nonumber
\end{align}
Combining (\ref{ineq:uniform_bound_additive_convex_2}) and (\ref{ineq:uniform_bound_additive_convex_3}), it follows that on an event with probability $1$, 
\begin{align*}
\limsup_{n \to \infty} M_n\leq C\big(\pnorm{\xi_1}{2}+\pnorm{H}{L_2(P_Z)}+\pnorm{\phi_0}{\infty}+\bar{M}\big),
\end{align*}
holds for some absolute constant $C>0$, thus proving that with probability $1$, 
\begin{align*}
\limsup_{n \to \infty}\bigabs{\inf_{x \in [1/4,3/4]} \hat{f}_n(x)}\leq C_{\xi, H,\phi_0,M}.
\end{align*}
That 
\begin{align*}
\limsup_{n \to \infty}\bigabs{\sup_{x \in [1/4,3/4]} \hat{f}_n(x)}\leq C'_{\xi, H,\phi_0,M}
\end{align*}
with probability $1$ can be proved in a completely similar manner by noting that the supremum of $\hat{f}_n$ over $[1/4,3/4]$ is taken either at $1/4$ or $3/4$. These claims show that with probability $1$, 
\begin{align*}
\limsup_{n \to \infty}\sup_{x \in [1/4,3/4]}\abs{\hat{f}_n(x)}\leq C''_{\xi,H,\phi_0,M}.
\end{align*}

Note that we have also verified the announced claim (\ref{ineq:upper_bound_cvx_origin_5}) in the convex case of Lemma \ref{lem:uniform_bound_iso} by taking $\phi_0(x,z)\equiv f_0(x)$ and $\mathcal{H}\equiv \{0\}$. The rest of proof for handling term (iii) proceeds along the same lines as in the proof of the convex case of Lemma \ref{lem:uniform_bound_iso}, modulo the unproved claim (\ref{ineq:uniform_bound_additive_convex_1}). Below we prove that (\ref{ineq:uniform_bound_additive_convex_1}) holds for $M>\sqrt{32\big(\pnorm{\xi_1}{2}^2+\pnorm{\phi_0}{\infty}^2+P_ZH^2\big)}$. To this end, first we prove
\begin{align}\label{ineq:uniform_bound_additive_convex_5}
\Prob\bigg(\mathcal{E}_1\equiv  \big\{ \inf_{(x,z) \in [1/4,3/4]^2} \big(\hat{\phi}_n(x,z)-\phi_0(x,z)\big)>M\textrm{ i.o.} \big\}\bigg)=0.
\end{align}
On the event $\mathcal{E}_1$ intersecting a probability-one event, there exists a subsequence $\{n_k\}_{k\geq 1}$ such that  
\begin{align}\label{ineq:uniform_bound_additive_convex_6}
&\liminf_{k \to \infty}\frac{1}{n_k}\sum_{i=1}^{n_k}\big(Y_i-\hat{\phi}_{n_k}(X_i,Z_i)\big)^2\\
&\geq \liminf_{k \to \infty}\frac{1}{2n_k}\sum_{(X_i,Z_i) \in [1/4,3/4]^2}\big(\phi_0-\hat{\phi}_{n_k}\big)^2(X_i,Z_i)-\lim_{k \to \infty} \frac{1}{n_k}\sum_{i=1}^{n_k} \xi_i^2\nonumber\\
& \geq M^2/8 - \E \xi_1^2,\nonumber
\end{align}
and thus by  (\ref{ineq:uniform_bound_additive_convex_3}), $M^2\leq 32\big(\pnorm{\xi_1}{2}^2+\pnorm{\phi_0}{\infty}^2+P_ZH^2\big)$. Hence $\mathcal{E}_1$ must be a probability-zero event, which proves (\ref{ineq:uniform_bound_additive_convex_5}). Using the same arguments we can prove
\begin{align}\label{ineq:uniform_bound_additive_convex_7}
\Prob\bigg(\sup_{(x,z) \in [1/4,3/4]^2} \big(\hat{\phi}_n(x,z)-\phi_0(x,z)\big)<-M \textrm{ i.o.} \bigg)=0.
\end{align}
The claim (\ref{ineq:uniform_bound_additive_convex_1}) now follows from (\ref{ineq:uniform_bound_additive_convex_5}) and (\ref{ineq:uniform_bound_additive_convex_7}). This completes the proof.
\end{proof}

\section*{Acknowledgements}
We thank Tengyao Wang for his generous help in the proof of Lemma \ref{lem:uniform_bound_iso}. 

\bibliographystyle{abbrv}
\bibliography{mybib}

\end{document}